\documentclass[twoside,11pt,final]{entics} 
\usepackage{enticsmacro}
\usepackage{graphicx}
\usepackage[all]{xy}

\usepackage{amsmath,amssymb}
\usepackage{stmaryrd,enumerate,mathpartir}
\usepackage{graphicx,graphbox,hyperref}
\usepackage{tikz-cd}
\usepackage{csquotes}
\usepackage{mdframed}
\allowdisplaybreaks
\newcommand{\X}{\mathbb{X}}
\newcommand{\A}{\mathbb{A}}

\newcommand{\bh}{\mathbf{H}}
\newcommand{\bv}{\mathbf{V}}

\newcommand{\getl}{\mathsf{getL}}
\newcommand{\putr}{\mathsf{putR}}
\newcommand{\getr}{\mathsf{getR}}
\newcommand{\putl}{\mathsf{putL}}

\newcommand{\corner}[1]{{\text{}^\ulcorner_\llcorner\!{#1}\!_\lrcorner^\urcorner}}
\newcommand{\ex}[1]{{#1}^{\circ\bullet}}

\newcommand{\cells}[4]{({\scriptstyle #1} {{\scriptstyle #2} \atop {\scriptstyle #3}} {\scriptstyle #4})}
\newcommand{\wiggle}[4]{\{{\scriptstyle #1} {{\scriptstyle #2} \atop {\scriptstyle #3}} {\scriptstyle #4}\}}
\newcommand{\floatcells}[4]{\left({\scriptstyle #1} {{\scriptstyle #2} \atop {\scriptstyle #3}} {\scriptstyle #4}\right)}
\newcommand{\floatwiggle}[4]{\left\{{\scriptstyle #1} {{\scriptstyle #2} \atop {\scriptstyle #3}} {\scriptstyle #4}\right\}}

\newcommand{\from}{\leftarrow}

\newcommand{\asn}{\downarrow}

\makeatletter
\providecommand{\leftsquigarrow}{%
  \mathrel{\mathpalette\reflect@squig\relax}%
}
\newcommand{\reflect@squig}[2]{%
  \reflectbox{$\m@th#1\rightsquigarrow$}%
}
\makeatother

\newcommand{\seq}{,\!\ldots\!,}

\makeatletter
\newcommand{\vast}{\bBigg@{4}}
\newcommand{\Vast}{\bBigg@{5}}
\makeatother

\newif\ifhideproofs

\ifhideproofs
\usepackage{environ}
\NewEnviron{hide}{}

\fi

\def\conf{MFPS 2026} 	
\volume{6}			
\def\volu{6}			
\def\papno{14}			
\def\mydoi{19573}
\def\lastname{Nester, Voorneveld} 
\def\runtitle{A Simple Categorical Calculus of Interacting Processes} 
\def\copynames{C. Nester, N. Voorneveld}    
\begin{document}
\begin{frontmatter}
  \title{A Simple Categorical Calculus of Interacting Processes} 						
 \thanks[ALL]{Chad Nester was supported by the Estonian Research Council grant PRG2764. Niels Voorneveld was supported by Estonian Research Council grant No. PRG1780.}   
  \author{Chad Nester\thanksref{a}\thanksref{myemail}}	
   \author{Niels Voorneveld\thanksref{b}\thanksref{coemail}}		
   \address[a]{University of Tartu\\				
    Tartu, Estonia}  							
   \thanks[myemail]{Email: \href{mailto:nester@ut.ee} {\texttt{\normalshape
        nester@ut.ee}}} 
  \address[b]{Cybernetica AS\\
    Tallinn, Estonia} 
  \thanks[coemail]{Email:  \href{mailto:niels.voorneveld@cyber.ee} {\texttt{\normalshape
        niels.voorneveld@cyber.ee}}}
\begin{abstract} 
  We present a calculus that models a simple sort of process interaction. Our calculus consists of a collection of terms together with a rewrite relation, parameterised by an arbitrary multicategory whose morphisms we understand as non-interactive processes. We show that our calculus is confluent and terminating, and that terms modulo the induced convertibility relation form a virtual double category. We relate our calculus to the free cornering of a monoidal category, which is a double-categorical model of process interaction that is similar in spirit to the calculus presented herein. Precisely, we construct a functor from the virtual double category given by our calculus into the underlying virtual double category of the free cornering of the free monoidal category on the multicategory of non-interacting processes. If we think of the terms of our calculus as programs and the rewriting system as an operational semantics for these programs, this functor gives a sound denotational semantics for our calculus in terms of the free cornering.
\end{abstract}
\end{frontmatter}

\section{Introduction}\label{sec:intro}
This paper concerns the categorical semantics of interaction. Specifically, we give a calculus of interacting processes and investigate its categorical structure. Our calculus is parameterised by a multicategory $\mathcal{M}$, the morphisms of which we think of as non-interactive processes. Our calculus augments these non-interactive processes with the ability to interact via a simple message-passing mechanism. More concretely, our calculus consists of a collection $T(\mathcal{M})$ of terms presented as a sequent calculus together with a rewrite relation $\to$ on those terms. This rewrite relation is confluent and terminating. When considered modulo the induced convertibility relation, the terms of our calculus form a virtual double category $[\mathcal{M}]$.

Our calculus is similar in spirit to the free cornering $\corner{\mathbb{A}}$ of a monoidal category $\mathbb{A}$~\cite{Nester2021a,Nester2023}. Both model the same sort of interaction, and are constructed from a pre-existing collection of non-interacting processes, being the monoidal category $\mathbb{A}$ in the case of the free cornering. Given this, it is not too surprising that the two formalisms are also related in a technical sense. The free cornering of a monoidal category is a strict double category, and there is a close relationship between strict double categories and virtual double categories, analogous to the relationship between strict monoidal categories and multicategories. Specifically, there are adjunctions:
\begin{mathpar}
  \begin{tikzcd}
    \mathsf{Mul} \ar[r,phantom,"\scriptstyle{\bot}"] \ar[r,shift left=0.5em,"F"]  & \mathsf{Mon} \ar[l,shift left=0.5em,"U"]
  \end{tikzcd}
  
  \begin{tikzcd}
    \mathsf{VDC} \ar[r,phantom,"\scriptstyle{\bot}"] \ar[r,shift left=0.5em,"F"]  & \mathsf{Dbl} \ar[l,shift left=0.5em,"U"]
  \end{tikzcd}
\end{mathpar}
where $\mathsf{Mul}$, $\mathsf{Mon}$, $\mathsf{VDC}$, and $\mathsf{Dbl}$ are the category of multicategories, strict monoidal categories, virtual double categories, and strict double categories, respectively. These two adjunctions are instances of the same general theorem concerning notions of generalised multicategory that arise from cartesian monads~\cite{Leinster2004}.

For any multicategory $\mathcal{M}$, there is a functor of virtual double categories: 
\[
  \llbracket - \rrbracket : [\mathcal{M}] \to U\left(\corner{F(\mathcal{M})}\right)
\]

This is to say that our term calculus admits a sound denotational semantics, where terms are interpreted as cells of the free cornering of $F(\mathcal{M})$. The primary difference between the calculus presented here and the free cornering construction is that our calculus is \emph{dynamic}, being presented as a term rewriting system, while the free cornering is \emph{static}, in the sense that it is presented in terms of equations.

\subsection{Organisation}
In Section~\ref{sec:free-cornering} we recall the free cornering of a monoidal category. This is done before anything else so that the string diagrams for the free cornering, which we find useful in conveying the intuition behind our term calculus, are available in the rest of our development. In Section~\ref{sec:term-calculus} we construct the terms $T(\mathcal{M})$ and rewrite relation $\to$ that comprise our calculus, and show that $\to$ is confluent and terminating. We also give the interpretation $\llbracket - \rrbracket$ of terms as cells of the free cornering and show that it is coherent with respect to our rewrite relation, in the sense that if $t \stackrel{*}{\leftrightarrow} s$ then $\llbracket t \rrbracket = \llbracket s \rrbracket$. In Section~\ref{sec:vdc} we show that terms modulo the convertibility relation $\stackrel{*}{\leftrightarrow}$ form a virtual double category, and that $\llbracket - \rrbracket$ extends to a functor of virtual double categories in the manner described above. We conclude in Section~\ref{sec:conclusion}. 

\subsection{Prerequisites}
The reader is assumed to have some familiarity with category theory, including adjunctions, monoidal categories, multicategories, and the basic ideas of categorical logic (see e.g.,~\cite{MacLane1971,Lambek1986,Shulman2016}). Moreover, familiarity with double categories and virtual double categories (see e.g.,~\cite{Leinster2004,Crutwell2010}) is necessary. Some knowledge of term rewriting is also required. We will use notions like confluence, termination, and local confluence together with a few classical results and techniques in the theory of term rewriting without introduction or comment. An excellent reference is~\cite{Baader1998}. We imagine that some level of awareness concerning the field of programming language semantics would be helpful, but is not, strictly speaking, necessary.

\subsection{Related Work}
Our work is inspired by Cockett and Pastro's logic of message passing~\cite{Cockett2009} and Wadler's work on the semantics of session types~\cite{Wadler2014}. Other significant antecedents along these lines include the work of Caires and Pfenning~\cite{Caires2010}, Honda~\cite{Honda1993}, and Bellin and Scott~\cite{Bellin1994}. 

The calculus presented here grew out of our attempts to provide a more ``operational'' account of the free cornering of a monoidal category, which was introduced by Nester~\cite{Nester2021a,Nester2023} and has been the subject of a series of papers exploring its use as a categorical model of process interaction~\cite{Nester2022,Boisseau2023,Nester2024}.

\section{The Free Cornering of A Monoidal Category}\label{sec:free-cornering}
In this section we recapitulate the construction of the free cornering of a monoidal category and its interpretation as a theory of interacting processes. The free cornering is a single-object strict double category, and the reader who is unfamiliar with double categories may wish to consult e.g., ~\cite{Leinster2004,Nester2024}.

Morphisms in monoidal categories often admit interpretations as resource-transforming processes (see e.g.,~\cite{Coecke2014}). The objects of the category in question are interpreted as collections of resources, with the unit $I$ denoting the empty collection and $A \otimes B$ denoting the collection consisting of the resources of both $A$ and $B$. Then a morphism $f : A \to B$ is interpreted as a process that transforms the resources of $A$ into those of $B$, composition is interpreted as sequencing, and the tensor product allows independent processes to occur simultaneously. For example if we have morphisms $\mathbf{bake} : \mathsf{Dough} \otimes \mathsf{Oven} \to \mathsf{Bread} \otimes \mathsf{Oven}$ and $\mathbf{knead} : \mathsf{Dough} \to \mathsf{Dough}$ representing the processes of baking and kneading dough, respectively, then the composite morphism $(\textbf{knead} \otimes 1_{\mathsf{Oven}});\textbf{bake}$ represents the combined process of kneading and then baking a given unit of dough. It is often helpful to picture such processes using string diagrams, as in:
\begin{equation} \label{diagram:baking}
  \includegraphics[height=2cm,align=c]{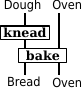} 
\end{equation}

The purpose of this section is to recall the \emph{free cornering} of a monoidal category, which is a strict single-object double category that models a simple sort of process interaction. If morphisms of a given monoidal category model processes of some kind, then cells of the free cornering model interacting processes of that kind. Interaction in the free cornering is governed by simple protocol types, constructed as follows:

\begin{definition}\label{def:exchanges}
  Let $X$ be a set. We define the monoid $\ex{X}$ of \emph{$X$-valued exchanges} to be the free monoid on the set of polarised elements of $X$, as in $\ex{X} = (X \times \{\circ,\bullet\})^*$. Explicitly, elements of $\ex{X}$ are sequences $A_1^{p_1}\cdots A_n^{p_n}$ with $A_i \in X$ and $p_i \in \{\circ,\bullet\}$ for each $1 \leq i \leq n$. We write $\lambda$ to indicate the empty sequence, and write either $UW$ or $U,W$ to indicate the concatenation of sequences $U$ and $W$.
\end{definition}

If elements of $X$ can be understood as resources, then elements of $\ex{X}$ can be understood as simple interaction protocols. Each protocol has two participants, one on the left and one on the right. The protocol $A^\circ$ demands that the left participant send the right participant an instance of the resource $A$, and dually $A^\bullet$ demands that the right participant send the left participant an instance of $A$. The protocol $UW$ demands that the participants carry out the protocol $U$, and then carry out the protocol $W$. The protocol $\lambda$ is the empty protocol, which demands nothing and is finished immediately. For example, suppose that $A,B \in X$, that the left participant is called ``Alice'', and that the right participant is called ``Bob''. In this case, to carry out $A^\circ B^\bullet A^\bullet$, first Alice must send Bob an instance of $A$, then Bob must send Alice an instance of $B$, and then Bob must send Alice an instance of $A$.

\begin{definition}\label{def:free-cornering}
  Let $\mathbb{A}$ be a strict monoidal category. The \emph{free cornering of $\mathbb{A}$}, written $\corner{\A}$, is the strict double category with a single object $\corner{\A}_0 = \{*\}$, with horizontal edge monoid $\corner{\A}_H = (\A_0,\otimes,I)$ given by the object monoid of $\A$, with vertical edge monoid given by $\corner{\A}_V = \ex{\A}_0$ given by the monoid of $\A_0$-valued exchanges (Definition~\ref{def:exchanges}), and with cells constructed according to the following rules:
  \begin{mathpar}
    \inferrule{f \in \A(A,B)}{\corner{f} : \floatcells{\lambda}{A}{B}{\lambda}}

    \inferrule{A \in \A_0}{{A\urcorner} : \floatcells{A^\circ}{I}{A}{\lambda}}

    \inferrule{A \in \A_0}{{A\llcorner} : \floatcells{\lambda}{A}{I}{A^\circ}}

    \inferrule{A \in \A_0}{{A\ulcorner} : \floatcells{\lambda}{I}{A}{A^\bullet}}

    \inferrule{A \in \A_0}{{A\lrcorner} : \floatcells{A^\bullet}{A}{I}{\lambda}}
    \\
    \inferrule{A \in \A_0}{1_A : \floatcells{\lambda}{A}{A}{\lambda}}

    \inferrule{a : \floatcells{U}{A}{B}{W} \\ b : \floatcells{U'}{B}{C}{W'}}{a \cdot b : \floatcells{UU'}{A}{C}{WW'}}
    
    \inferrule{U \in \ex{\A}_0}{id_U : \floatcells{U}{I}{I}{U}}
    
    \inferrule{a : \floatcells{U}{A}{B}{W} \\ b : \floatcells{W}{A'}{B'}{V}}{a \mid b : \floatcells{U}{A \otimes A'}{B \otimes B'}{V}}
  \end{mathpar}
  Cells are subject to a number of equations. First, equations concerning the cells $\corner{f}$:
  \begin{mathpar}
    \corner{f} \cdot \corner{g} = \corner{fg}

    \corner{f} \mid \corner{g} = \corner{f \otimes g}

    1_A = \corner{1_A}
  \end{mathpar}
  second, the \emph{yanking equations}:
  \begin{mathpar}
    {A\llcorner} \mid {A\urcorner} = 1_A

    {A\ulcorner} \mid {A\lrcorner} = 1_A

    {A\urcorner} \cdot {A\llcorner} = id_{A^\circ}
    
    {A\lrcorner} \cdot {A\ulcorner} = id_{A^\bullet}
  \end{mathpar}
  and finally equations ensuring that the axioms of a double category are satisfied:
  \begin{mathpar}
    1_A \cdot a = a = a \cdot 1_B
    
    (a \cdot b) \cdot c = a \cdot (b \cdot c)

    id_U \mid a = a = a \mid id_W

    (a \mid b) \mid c = a \mid (b \mid c)
    \\
    (a \mid b) \cdot (c \mid d) = (a \cdot c) \mid (b \cdot d)
    
    id_\lambda = 1_I

    id_U \cdot id_W = id_{UW}
  \end{mathpar}
  Note in particular that $1_{A \otimes B} = \corner{1_{A \otimes B}} = \corner{1_A \otimes 1_B} = \corner{1_A} \mid \corner{1_B} = 1_A \otimes 1_B$ already holds.
\end{definition}

If a morphism of $\A$ can be understood as a process that transforms the resources indicated by its domain into those indicated by its codomain, then a cell of $\corner{\A}$ can be understood as a process that transforms the resources indicated by its top boundary into those indicated by its bottom boundary by interacting with other such processes in the manner indicated by its left and right boundary. The cells $\corner{f}$ and the equations concerning them serve to include the processes represented by $\A$ in the new setting in a coherent manner. The \emph{corner cells} ${A\urcorner}$,${A\llcorner}$,${A\ulcorner}$, and ${A\lrcorner}$ allow interacting processes to exchange resources. It is helpful to think of the corner cells in terms of string diagrams, as in:
\begin{mathpar}
  A\urcorner
  \hspace{0.3cm}
  \leftrightsquigarrow
  \hspace{0.3cm}
  \includegraphics[height=0.77368cm, align=c]{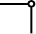}
  
  A\llcorner
  \hspace{0.3cm}
  \leftrightsquigarrow
  \hspace{0.3cm}
  \includegraphics[height=0.77368cm, align=c]{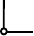}
  
  A\ulcorner
  \hspace{0.3cm}
  \leftrightsquigarrow
  \hspace{0.3cm}
  \includegraphics[height=0.77368cm, align=c]{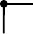}
  
  A\lrcorner
  \hspace{0.3cm}
  \leftrightsquigarrow
  \hspace{0.3cm}
  \includegraphics[height=0.77368cm, align=c]{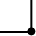}
\end{mathpar}
For example $A\llcorner$ is the process that transforms an instance of $A$ into nothing by giving it away along the right boundary, and $A\urcorner$ is the process that transforms nothing into an instance of $A$ by receiving that instance of $A$ along the left boundary. The yanking equations tell us that being exchanged in this manner has no effect on the resources involved. The name ``yanking equations'' comes from their string-diagrammatic representation, as in:
\begin{mathpar}
  \includegraphics[height=1.2cm,align=c]{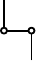}
  \hspace{0.25cm}
  \equiv
  \hspace{0.25cm}
  \includegraphics[height=1.2cm,align=c]{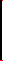}
  
  \includegraphics[height=1.2cm,align=c]{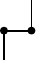}
  \hspace{0.25cm}
  \equiv
  \hspace{0.25cm}
  \includegraphics[height=1.2cm,align=c]{figs/c-2-5.pdf}

  \includegraphics[height=0.694734cm,align=c]{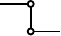}
  \hspace{0.25cm}
  \equiv
  \hspace{0.25cm}
  \includegraphics[width=1.2cm,align=c]{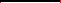}
    
  \includegraphics[height=0.694734cm,align=c]{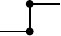}
  \hspace{0.25cm}
  \equiv
  \hspace{0.25cm}
  \includegraphics[width=1.2cm,align=c]{figs/c-2-6.pdf}
\end{mathpar}

For example, our procedure for baking bread (\ref{diagram:baking}) can be decomposed in the free cornering of the ambient monoidal category into the two cells pictured below left. Composing these cells horizontally recovers the image under $\corner{-}$ of the original morphism, pictured below on the right. 
\begin{mathpar}
  \includegraphics[height=2.284cm,align=c]{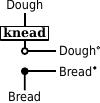} 

  \includegraphics[height=2.284cm,align=c]{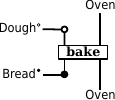} 

  \includegraphics[height=2.82cm,align=c]{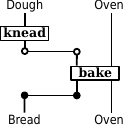} 
  \hspace{0.3cm}
  =
  \hspace{0.3cm}
  \includegraphics[height=2.82cm,align=c]{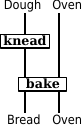} 
\end{mathpar}

\section{The Term Calculus}\label{sec:term-calculus}
In this section we introduce the terms and rewrite relation that make up our calculus. Our development takes place with respect to a fixed multicategory $\mathcal{M}$, whose morphisms we understand as (non-interactive) processes. The purpose of the calculus is to allow such processes to interact by exchanging resources along a left and right boundary, much as in the free cornering of a monoidal category. The major difference between the calculus presented here and the free cornering is that our calculus is \emph{directed}, with the relationship between terms given by a confluent and terminating rewrite relation. The free cornering is \emph{undirected}, with the relationship between terms given by equations. Put another way, while the free cornering is a \emph{static} model of process interaction, the term calculus presented here is a \emph{dynamic} model of process interaction.

\subsection{Sequent Calculus for Multicategories}\label{subsec:mult-sequent}
Our terms will be presented as a sequent calculus, and in light of the fact that the construction is parameterised by our fixed multicategory $\mathcal{M}$ it will be convenient to use the sequent calculus notation for multicategories, which we recall now. Let $\Sigma$ be a (multi-sorted) signature in which operation symbols are typed as in $f : A_1\seq A_n \vdash B$ where $A_1 \seq A_n,B$ are generating sorts. Then a \emph{term} over $\Sigma$ is a sequent that is derivable via the following inference rules:
\begin{mathpar}
  \inferrule*[right=var]{\text{ }}{x:A \vdash x:A}

  \inferrule*[right=op]{(\Gamma_i \vdash t_i : A_i)_{i =1}^{n} \\ (f : A_1 \seq A_n \vdash B) \in \Sigma}{\Gamma_1 \seq \Gamma_n \vdash f(t_1 \seq t_n) : B}
\end{mathpar}

Crucially, for a sequent $\Gamma \vdash t : B$ to be considered well-formed, the context $\Gamma$ must not contain any repeated variables. This means that, for example, when we write $\Gamma_1 \seq \Gamma_n$ it is implied that the variables in the $\Gamma_i$ are disjoint. Terms over $\Sigma$ form a multicategory, with identity morphisms given by the \textsc{var} rule and with composition given by substitution. More precisely, the composition operation is given by the following admissible inference rule:
\begin{mathpar}
  \inferrule*[right=comp]{(\Gamma_i \vdash t_i : A_i)_{i =1}^n \\ x_1 {:} A_1 \seq x_n {:} A_n \vdash t : B}{\Gamma_1\seq \Gamma_n \vdash t[t_1 \seq t_n/x_1 \seq x_n] : B}
\end{mathpar}
That this satisfies the equations of a multicategory follows from certain elementary properties of substitution. For example, for any $\Gamma \vdash t : B$ the right-unitality law ($1_B \circ f = f$) holds as in:
\[
  (\Gamma \vdash x[t/x] : B) = (\Gamma \vdash t : B)
\]

The multicategory of terms over a signature is in fact the \emph{free} multicategory over that signature, in that this construction gives the left adjoint of an adjunction between a category of signatures and the category of multicategories. The right adjoint maps a multicategory $\mathcal{M}$ to the $\mathcal{M}_0$-sorted signature $\Sigma_\mathcal{M}$ with an operation symbol $f : A_1,\ldots,A_n \vdash B$ for each $f \in \mathcal{M}(A_1,\ldots,A_n;B)$. The counit gives a morphism from the multicategory of terms over this signature into $\mathcal{M}$, and quotienting the terms by the equations that hold in the image of this morphism yields a sequent calculus presentation of $\mathcal{M}$. To understand the effect of these extra equations, suppose $f \in \mathcal{M}(A,A;B)$, $g \in \mathcal{M}(C;A)$ and consider the following derivations:
\begin{mathpar}
{\small 
  \inferrule*[right=op]
  {
    \inferrule*[right=var]
    {\text{ }}
    {x_1 : C \vdash x_1 : C}
    \\
    \inferrule*[right=var]
    {\text{ }}
    {x_2 : C \vdash x_2 : C}
    \\
    ((f \circ  (g,g)) : C,C \vdash B) \in \Sigma_\mathcal{M}
  }
  {x_1 {:} C,x_2 {:} C \vdash (f \circ (g,g))(x_1,x_2) : B}
}

{\small 
  \inferrule*[right=op]
  {
    \inferrule*[right=op]
    {
      \inferrule*
      {\text{ }}
      {x_1 : C \vdash x_1 : C}
      \\
      (g : C \vdash A) \in \Sigma_\mathcal{M}
    }
    {x_1 : C \vdash g(x_1) : A}
    \\
    \inferrule*[right=op]
    {
      \inferrule*
      {\text{ }}
      {x_2 : C \vdash x_2 : C}
      \\
      (g : C \vdash A) \in \Sigma_\mathcal{M}
    }
    {x_2 : C \vdash g(x_2) : A}
    \\
    (f : A,A \vdash B) \in \Sigma_\mathcal{M}
  }
  {x_1 {:} C,x_2 {:} C \vdash f(g(x_1),g(x_2)) : B}
}
\end{mathpar}
These both represent the same morphism of $\mathcal{M}$ in the sense that they are equal in the image of the counit, but are not equal in the free multicategory over the underlying signature $\Sigma_{\mathcal{M}}$ of $\mathcal{M}$. To obtain a sequent calculus presentation of $\mathcal{M}$, we must reconcile these two ways of representing composition (and identities) in $\mathcal{M}$, which is achieved by quotienting our derivations in the manner described above.

It follows that we may reason about morphisms in any multicategory by means of a sequent calculus, interpreting composition as substitution and identities as variables. Going forward, we will write $\Gamma \vdash v : A$ to indicate a morphism of our fixed multicategory $\mathcal{M}$. For a more detailed account of the connection between multicategories and sequent calculi see e.g.,~\cite{Shulman2016}.

\subsection{Terms}
We are now ready to introduce the terms of our calculus:
\begin{definition}\label{def:terms}
  The collection of \emph{terms over $\mathcal{M}$}, written $T(\mathcal{M})$, is constructed according to the inference rules of Figure~\ref{fig:terms}. We retain the convention that contexts $\Gamma$ must not contain repeated variables.

  \begin{figure}\begin{mdframed}
  \begin{center}
    \begin{mathpar}
      \inferrule*[right=seq]{\Gamma \vdash v : A}{\Gamma \Vdash [v] : (\lambda,A,\lambda)}

      \inferrule*[right=let]{\left(\Gamma_i \Vdash t_i : (U_{i-1},A_i,U_i)\right)_{i = 1}^n \\ x_1 {:} A_1 \seq x_n {:} A_n \Vdash t : (V,B,W)}{\Gamma_1 \seq \Gamma_n \Vdash \textsf{let } x_1 \seq x_n \asn (t_1 \mid \cdots \mid t_n) \textsf{ in } t : (U_0V,B,U_n W)}

      \inferrule*[right=putr]{\Delta \Vdash t : (U,B,W) \\ \Gamma \vdash v : A}{\Delta,\Gamma \Vdash \putr(v,t) : (U,B,A^\circ W)}

      \inferrule*[right=getl]{x:A,\Gamma \Vdash t : (U,B,W)}{\Gamma \Vdash \getl(x.t) : (A^\circ U,B,W)}
      \\
      \inferrule*[right=getr]{\Gamma,x:A \Vdash t : (U,B,W)}{\Gamma \Vdash \getr(x.t) : (U,B,A^\bullet W)}

      \inferrule*[right=putl]{\Gamma \vdash v : A \\ \Delta \Vdash t : (U,B,W)}{\Gamma,\Delta \Vdash \putl(v,t) : (A^\bullet U,B,W)}
    \end{mathpar}
  \end{center}
\end{mdframed}
\caption{Term formation rules for $T(\mathcal{M})$.}\label{fig:terms}
\end{figure}
\end{definition}

In what follows terms $t$ are implied to be part of the ultimate conclusion $\Gamma \Vdash t : (U,B,W)$ of some derivation. We will often write the term $t$ to stand for its derivation, and will often speak of the two interchangeably. This is, in some sense, the point of terms: to be a sort of shorthand for the accompanying derivation. For example, while our rewrite relation is specified on terms, it should be understood as rewriting the associated derivations. We introduce a bit of useful terminology:
\begin{definition}
  We say that a term is:
  \begin{itemize}
  \item \emph{neutral}, or sometimes a \emph{value}, if it is of the form $[v]$.
  \item \emph{left-facing} if it is of the form $\putl(v,t)$ or $\getl(x.t)$.
  \item \emph{right-facing} if it is of the form $\putr(v,t)$ or $\getr(x.t)$.
  \item a \emph{let-binding} if it is of the form $\textsf{let } x_1 \seq x_n \asn (t_1 \mid \cdots \mid t_n) \textsf{ in } t$.
  \end{itemize}
\end{definition}

If morphisms of $\mathcal{M}$ admit interpretation as resource-transforming processes, then terms of $T(\mathcal{M})$ can be understood as \emph{interacting} resource-transforming processes, much as in the free cornering of a monoidal category. In general, a derivable sequent $\Gamma \Vdash t : (U,B,W)$ is understood as a process that transforms the resources indicated by $\Gamma$ into those indicated by $B$ by interacting with other such processes in the manner indicated by $U$ and $W$. More precisely, neutral terms $[v]$ perform no interaction, and represent the same process that $v$ does in $\mathcal{M}$. The term $\putr(v,t)$ represents the process in which we perform $v$ on some inputs, send the result out along the right boundary, and proceed as in $t$. The term $\getr(x.t)$ represents the process in which we wait for some input along the right boundary, and proceed as in $t$ once we have it. The left-facing terms are interpreted similarly. Finally, the term $\textsf{let } x_1 \seq x_n \asn (t_1 \mid \cdots \mid t_n) \textsf{ in } t'$ is the process in which we perform the processes indicated by the $t_i$ and, once they are all finished, use the resulting values as the input to the process $t'$. The processes indicated by the $t_i$ may interact as they are performed by exchanging resources along their left and right boundaries via the $\textsf{get}$ and $\textsf{put}$ constructors.

This is similar to the way in which cells of the free cornering can be understood as interacting processes, raising the question of how the two formalisms relate. We answer this question by giving an interpretation $\llbracket - \rrbracket$ of terms in $T(\mathcal{M})$ as cells of  the free cornering of the free monoidal category $F(\mathcal{M})$ on $\mathcal{M}$. The object monoid of $F(\mathcal{M})$ is the free monoid on $\mathcal{M}_0$, and a morphism of $F(\mathcal{M})(\Gamma_1 \seq \Gamma_n;A_1\seq A_n)$ is a sequence $(f_1\seq f_n)$ of morphisms of $\mathcal{M}$ where $f_i \in \mathcal{M}(\Gamma_i;A_i)$ for each $i \in \{1 \seq n\}$. Composition is defined in terms of composition in $\mathcal{M}$, which works because there is at most one way to compose suitably typed sequences $(f_1\seq f_n)$ and $(g_1 \seq g_m)$. Identities are given by sequences of identity morphisms, and the monoidal structure is given by concatenation of sequences with the empty sequence serving as the unit. For more details on the adjunction relating multicategories and strict monoidal categories see e.g.,~\cite{Leinster2004}.

Specifically, we define:
\begin{definition}\label{def:interpretation}
Let $x_1 : A_1 \seq x_n : A_n \Vdash t : (U,B,W)$ be a term of $T(\mathcal{M})$. Then the cell $\llbracket t \rrbracket : \cells{U}{A_1,\ldots,A_n}{B}{W}$ of $\corner{F(\mathcal{M})}$ is defined as in: 
\[
  \llbracket t \rrbracket =
  \begin{cases}
    \corner{(v)} &\text{if } t = [v]\\
    (\llbracket t_1 \rrbracket \mid \cdots \mid \llbracket t_m \rrbracket) \cdot \llbracket t' \rrbracket &\text{if } t = \textsf{let } y_1,\ldots,y_m \asn (t_1 \mid \cdots \mid t_m) \textsf{ in } t'\\
    ((\corner{(v)} \cdot A\lrcorner) \mid 1_\Delta) \cdot \llbracket t' \rrbracket &\text{if } t = \putl(v,t')\\
    (A\urcorner \mid 1_\Gamma) \cdot \llbracket t' \rrbracket &\text{if } t = \getl(x.t')\\
    (1_\Gamma \mid A\ulcorner) \cdot \llbracket t' \rrbracket &\text{if  } t = \getr(x.t')\\
    (1_\Delta  \mid (\corner{(v)} \cdot A\llcorner)) \cdot \llbracket t' \rrbracket  &\text{if } t = \putr(v,t')
  \end{cases}
\]
This can, alternatively, be expressed in terms of string diagrams as in Figure~\ref{fig:interpretation}. We find these diagrams to be a helpful tool in developing one's intuition about the term calculus presented herein.

\begin{figure}
  \begin{mdframed}
    \begin{mathpar}
      \llbracket [v] \rrbracket
      \hspace{0.3cm}\leftrightsquigarrow\hspace{0.3cm}
      \includegraphics[height=1.2cm,align=c]{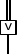} 

      \llbracket \textsf{let } x_1,\ldots,x_n \asn (t_1\mid \cdots \mid t_n) \textsf{ in } t' \rrbracket 
      \hspace{0.3cm}\leftrightsquigarrow\hspace{0.3cm}
      \includegraphics[height=1.2cm,align=c]{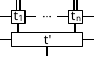} 
      \\
      \llbracket \putr(v,t') \rrbracket
      \hspace{0.3cm}\leftrightsquigarrow\hspace{0.3cm}
      \includegraphics[height=1.5cm,align=c]{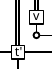} 

      \llbracket \getl(x.t') \rrbracket
      \hspace{0.3cm}\leftrightsquigarrow\hspace{0.3cm}
      \includegraphics[height=1.2cm,align=c]{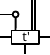} 
      \\
      \llbracket \getr(x.t') \rrbracket
      \hspace{0.3cm}\leftrightsquigarrow\hspace{0.3cm}
      \includegraphics[height=1.2cm,align=c]{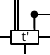} 

      \llbracket \putl(v,t') \rrbracket
      \hspace{0.3cm}\leftrightsquigarrow\hspace{0.3cm}
      \includegraphics[height=1.5cm,align=c]{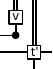} 
    \end{mathpar}
   \end{mdframed}
    \caption{Interpretation of terms in $T(\mathcal{M})$ as cells of $\corner{F(\mathcal{M})}$.}\label{fig:interpretation}
\end{figure}
\end{definition}

Next, we define a substitution operation, called \emph{value substitution}, on terms of $T(\mathcal{M})$. Value substitution is ultimately resolved in terms of composition in $\mathcal{M}$, via the admissible \textsc{comp} rule in the associated sequent calculus (Section~\ref{subsec:mult-sequent}). It is convenient to define value substitution by manipulating derivations explicitly, from which perspective we are introducing an admissible inference rule in the sequent calculus presenting $T(\mathcal{M})$. The admissible rule is:
\begin{mathpar}
  \inferrule*[right=vsub]{(\Gamma_i \vdash v_i : A_i)_{i=1}^{n}  \\ x_1  {:} A_1 \seq x_n {:} A_n \Vdash t : (U,B,W)}{\Gamma_1 \seq \Gamma_n \Vdash t[v_1,\ldots,v_n/x_1,\ldots,x_n] : (U,B,W)}
\end{mathpar}
To show that \textsc{vsub} is admissible is equivalently to define the operation of value substitution. We do so by induction on the form of $t$. The cases are as follows: If $t = [u]$, we define:
\begin{mathpar}
  \inferrule*[right=vsub]
  {
    (\Gamma_i \vdash v_i : A_i )_{i = 1}^n
    \\
    \inferrule*[right=seq]
    {x_1 {:} A_1 \seq x_n {:} A_n \vdash  u : B }
    {x_1 {:} A_1 \seq x_n {:} A_n \Vdash [u] : (\lambda,B,\lambda)}
  }
  {\Gamma_1 \seq \Gamma_n \Vdash [u][v_1 \seq v_n/x_1 \seq x_n] : (\lambda,B,\lambda)}
  \\
  = 
  \\
  \inferrule*[right=seq]
  {
    \inferrule*[right=comp]
    {(\Gamma_i \vdash v_i : A_i)_{i = 1}^n \\ x_1 {:} A_1 \seq x_n {:} A_n \vdash u : B}
    {\Gamma_1 \seq \Gamma_n \vdash u[v_1 \seq v_n/x_1 \seq x_n] : B}
  }
  {\Gamma_1 \seq \Gamma_n \Vdash [u[v_1 \seq v_n/x_1 \seq x_n]] : (\lambda,B,\lambda)}
\end{mathpar}
If $t = \textsf{let } y_1 \seq y_m \asn (t_1\mid \cdots \mid t_m) \textsf{ in } r$, we define:
{\small
\begin{mathpar}
  \inferrule*[right=vsub]
  {
    ((\Gamma_j^i \vdash v_j^i : A_j^i)_{i=1}^{k_j})_{j=1}^m
    \\
    \inferrule*[right=let]
    {
      (x_j^1{:}A_j^1 \seq x_j^{k_j}{:}A_j^{k_j} \Vdash t_j : (U_{j-1},B_j,U_j))_{j=1}^m
      \\
      y_1 {:} B_1 \seq y_m {:} B_m \Vdash r : (U,B,W)
    }
    {x_1^1 {:} B_1^1 \seq x_m^{k_m} {:} B_m^{k_m} \Vdash \textsf{let } y_1 \seq y_m \asn (t_1 \mid \cdots \mid t_m) \textsf{ in } r : (U_0U,B,U_mW)}
  }
  {\Gamma_1^1 \seq \Gamma_m^{k_m} \Vdash (\textsf{let } y_1 \seq y_m \asn (t_1 \mid \cdots \mid t_m) \textsf{ in } r)[v_1^1 \seq v_m^{k_m}/x_1^1 \seq x_m^{k_m}] : (U_0U,B,U_mW)}
  \\
  =
  \\
  \inferrule*[right=let]
  {
    \left(
        \inferrule
        {
          (\Gamma_j^i \vdash v_j^i)_{i=1}^{k_j}
          \\
          x_j^1{:}A_j^1 \seq x_j^{k_j} {:} A_j^{k_j} \Vdash t_j : (U_{j-1},B_j,U_j)
        }
        {\Gamma_j^1 \seq \Gamma_j^{k_j} \Vdash t_j[v_j^1 \seq v_j^{k_j} / x_j^1 \seq x_j^{k_j}] : (U_{j-1},B_j,U_j)} 
        {\tiny \text{VSUB} \tiny} 
  \right)_{j=1}^m
  \\
    y_1{:}B_1 \seq y_m{:}B_m \Vdash r : (U,B,W)
  }
  {\Gamma_1^1 \seq \Gamma_m^{k_m} \Vdash \textsf{let } y_1 \seq y_m \asn (t_1[v_1^1 \seq v_1^{k_1}/x_1^1 \seq x_1^{k_1}] \mid \cdots \mid t_m[v_m^1 \seq v_m^{k_m} / x_m^1 \seq x_m^{k_m}]) \textsf{ in } r : (U_0U,B,U_mW)}
\end{mathpar}}

If $t = \putr(v,s)$ we define:
{\small \begin{mathpar}
    \inferrule*[right=vsub]
    {
      (\Gamma_i \vdash v_i : B_i)_{i = 1}^n \\ (\Delta_j \vdash u_j : A_j)_{j=1}^m \\
      \inferrule*[right=putr]
      {x_1 {:} B_1 \seq x_n {:} B_n \Vdash s : (U,B,W) \\ y_1 {:} A_1 \seq y_m {:} A_m \vdash v : A}
      {x_1 {:} B_1 \seq x_n {:} B_n,y_1 {:} A_1 \seq y_m {:} A_m \Vdash \putr(v,s) : (U,B,A^\circ W)}
      }
      {\Gamma_1 \seq \Gamma_n,\Delta_1,\ldots,\Delta_m \Vdash \putr(v,s)[v_1 \seq v_n,u_1 \seq u_m/x_1 \seq x_n,y_1 \seq y_m] : (U,B,A^\circ W) }
      \\
      =
      \\
      \inferrule*[right=putr]
      {
        \inferrule*[right=vsub]
        {(\Gamma_i \vdash v_i : B_i)_{i = 1}^n \\ x_1 {:} B_1 \seq x_n {:} B_n \Vdash s : (U,B,W)}
        {\Gamma_1 \seq \Gamma_n \Vdash s[v_1 \seq v_n/x_1 \seq x_n] : (U,B,W)}
        \\
        \inferrule*[right=comp]
        {(\Delta_j \vdash u_j : A_j)_{j=1}^m \\ y_1 {:} A_1 \seq y_m {:} A_m \vdash v : A}
        {\Delta_1 \seq \Delta_m \vdash v[u_1\seq u_m / y_1 \seq y_m] : A}
      }
      {\Gamma_1 \seq \Gamma_n,\Delta_1 \seq \Delta_m \Vdash \putr(v[u_1 \seq u_n/y_1 \seq y_n],s[v_1 \seq v_m / x_1 \seq x_m]) : (U,B,A^\circ W)}
    \end{mathpar}}

  The case for $t = \putl(v,s)$ is similar. If $t = \getr(x.s)$ we define:
  \begin{mathpar}
  \inferrule*[right=vsub]
    {
      (\Gamma_i \vdash v_i : A_i)_{i = 1}^n
      \\
      \inferrule*[right=getr]
      {x_1 {:} A_1 \seq x_n {:} A_n,x {:} A \Vdash s : (U,B,W)}
      {x_1 {:} A_1 \seq x_n {:} A_n \Vdash \getr(x.s) : (U,B,A^\bullet W)}
    }
    {\Gamma_1 \seq \Gamma_n \Vdash \getr(x.s)[v_1 \seq v_n / x_1 \seq x_n] : (U,B,A^\bullet W)}
    \\
    =
    \\
    \inferrule*[right=getr]
    {
      \inferrule*[right=vsub]
      {
        (\Gamma_i \vdash v_i : A_i)_{i = 1}^n
        \\
        \inferrule*[right=id]{\text{ }}{x : A \vdash x : A}
        \\
        x_1 {:} A_1 \seq x_n {:} A_n,x {:} A \Vdash s : (U,B,W)
      }
      {\Gamma_1 \seq \Gamma_n,x {:} A \Vdash s[v_1 \seq v_n,x/x_1 \seq x_n,x] : (U,B,W)}
    }
    {\Gamma_1 \seq \Gamma_n \Vdash \getr(x.s[v_1 \seq v_n,x/x_1 \seq x_n,x]) : (U,B,A^\bullet W)}
  \end{mathpar}
  The case for $t = \getl(x.s)$ is similar, and it follows that $\textsc{vsub}$ is admissible.

  We will sometimes omit variables that are to be substituted for themselves from the substitution list, so that for example $t[x_1\seq v \seq x_n / x_1 \seq x \seq x_n] = t[v/x]$ and $t[v_1 \seq v_n,x / x_1 \seq x_n,x] = t[v_1\seq v_n/x_1 \seq x_n]$. We may also abbreviate lists as in $\overline{v} = v_1\seq v_n$ when the meaning of $\overline{v}$ is clear in context, so that for example we may write $t[\overline{v}/\overline{x}] = t[v_1\seq v_n/x_1 \seq x_n]$. Then a shorter definition of value substitution is:
  \[
    t[\overline{v}/\overline{x}] =
    \begin{cases}
      [u[\overline{v}/\overline{x}]] &\text{if } t = [u]\\
      \putr(u[\overline{v_2}/\overline{x_2}],s[\overline{v_1}/\overline{x_1}]) &\text{if } t = \putr(u,s) \text{ and } [\overline{v}/\overline{x}] = [\overline{v_1},\overline{v_2}/\overline{x_1},\overline{x_2}]\\
      \getl(y.s[\overline{v}/\overline{x}]) &\text{if } t = \getl(y.s)\\
      \putl(u[\overline{v_1}/\overline{x_1}],s[\overline{v_2}/\overline{x_2}]) &\text{if } t = \putr(u,x) \text{ and } [\overline{v}/\overline{x}] = [\overline{v_1},\overline{v_2}/\overline{x_1},\overline{x_2}]\\
      \getr(y.s[\overline{v}/\overline{x}]) &\text{if } t = \getr(y.s)\\
      \textsf{let } y_1 \seq y_m \asn (t_1[\overline{v_1}/\overline{x_1}] \mid \cdots \mid t_n[\overline{v_n}/\overline{x_n}]) \textsf{ in } s &\text{if } t = \textsf{let } y_1 \seq y_m \asn (t_1 \mid \cdots \mid t_n) \textsf{ in } s \\&\text{and } [\overline{v} / \overline{x}] = [\overline{v_1} \seq \overline{v_n}/\overline{x_1}\seq \overline{x_n}]\\
    \end{cases}
  \]
  There will only ever be one well-defined way to split up the variables $\overline{x}$ in the cases that require it, as can be seen in the more explicit definition of value substitution in terms of derivations.

\subsection{Rewrite Rules}
  We proceed to give a term rewrite relation on $T(\mathcal{M})$.
\begin{definition}\label{def:rewrites}
  Let $\to$ be the least congruence on terms of $T(\mathcal{M})$ generated by the rewrites of Figure~\ref{fig:rewrites}. 

  \begin{figure}\begin{mdframed}
      \begin{itemize}
      \item[] \textbf{[R0]} $\textsf{let } x_1 \seq x_n \asn ([v_1] \mid \cdots \mid [v_n]) \textsf{ in } t \,\to\, t[v_1 \seq v_n/x_1 \seq x_n]$
      \item[] \textbf{[R1]} $\textsf{let } x_1 \seq x_n \asn (\overline{a} \mid \putr(v,t) \mid \getl(y.s) \mid \overline{b}) \textsf{ in } r \,\to\, \textsf{let } x_1 \seq x_n \asn (\overline{a} \mid t \mid s[v/y] \mid \overline{b}) \textsf{ in } r$
      \item[] \textbf{[R2]} $\textsf{let } x_1 \seq x_n \asn (\overline{a} \mid \getr(y.t) \mid \putl(v,s) \mid \overline{b}) \textsf{ in } r \,\to\, \textsf{let } x_1 \seq x_n \asn (\overline{a} \mid t[v/y] \mid s \mid \overline{b}) \textsf{ in } r$
      \item[] \textbf{[R3]} $\textsf{let } x_1 \seq x_n \asn (\putl(v,t) \mid \overline{a}) \textsf{ in } r \,\to\, \putl(v,\textsf{let } x_1 \seq x_n \asn (t \mid \overline{a}) \textsf{ in } r)$
      \item[] \textbf{[R4]} $\textsf{let } x_1 \seq x_n \asn (\getl(y.t) \mid \overline{a}) \textsf{ in } r \,\to\, \getl(y.\textsf{let } x_1 \seq x_n \asn (t \mid \overline{a}) \textsf{ in } r)$
      \item[] \textbf{[R5]} $\textsf{let } x_1 \seq x_n \asn (\overline{a} \mid \putr(v,t)) \textsf{ in } r \,\to\, \putr(v,\textsf{let } x_1 \seq x_n \asn (\overline{a} \mid t) \textsf{ in } r)$
      \item[] \textbf{[R6]} $\textsf{let } x_1 \seq x_n \asn (\overline{a} \mid \getr(y.t)) \textsf{ in } r \,\to\, \getr(y.\textsf{let } x_1 \seq x_n \asn (\overline{a} \mid t) \textsf{ in } r)$
      \item[] \textbf{[R7]} $\putr(v,\putl(w,t)) \,\to\, \putl(w,\putr(v,t))$
      \item[] \textbf{[R8]} $\putr(v,\getl(x.t)) \,\to\, \getl(x.\putr(v,t))$ when $x$ does not occur in $v$
      \item[] \textbf{[R9]} $\getr(x.\putl(v,t)) \,\to\, \putl(v,\getr(x.t))$ when $x$ does not occur in $v$
      \item[] \textbf{[R10]} $\getr(x.\getl(y.t)) \,\to\, \getl(y.\getr(x.t))$
      \end{itemize}
  \end{mdframed}\caption{Generating rewrites for $\to$.}\label{fig:rewrites}\end{figure}
\end{definition}
\noindent
Here we use $\overline{a}$ as a shorthand for $a_1 \mid \dots \mid a_n$ as well, that is, a list of parallel terms in a let-binding.

Adopting the usual notation for term rewriting relations, we write $\stackrel{*}{\to}$ for the reflexive transitive closure of $\to$, write $\stackrel{*}{\leftrightarrow}$ for the symmetric and transitive closure of $\stackrel{*}{\to}$, and write $\from$ and $\stackrel{*}{\from}$ for the converse of $\to$ and reflexive transitive closure of $\from$, respectively. Intuitively, rule \textbf{[R0]} says that when all of the inputs to a let-binding are neutral we may substitute them into the term in which they are bound via our value substitution operation \textsc{vsub}. Rules \textbf{[R1]} and \textbf{[R2]} say that when matching left-facing and right-facing terms occur next to each other in a let-binding they may interact, and it is through these two rules that processes interact in our calculus. Rules \textbf{[R3-R6]} allow ``outward-facing'' gets and puts in the input sequence of a let-binding to be moved ``above'' the let-binding, and rules \textbf{[R7-R10]} allow us to commute left-facing gets and puts past right-facing ones. It may be helpful to conceptually separate \textbf{[R0-R6]}, which perform useful computational work, from \textbf{[R7-R10]}, which serve mainly to ensure that $\to$ is confluent. We note that the restriction on the use of \textbf{[R8]} and \textbf{[R9]} is not strictly necessary: if the reduct is well-formed then typing constraints ensure that $x$ cannot occur in $v$ anyway. 

Together, the terms of Definition~\ref{def:terms} and rewrites of Definition~\ref{def:rewrites} constitute a \emph{calculus}. We give a few simple examples demonstrating the way in which this calculus models process interaction:
\begin{example}
  Say $\mathcal{M}$ has objects including $\mathbf{Pants}$, $\mathbf{Shirt}$, $\mathbf{Clothes}$, $\mathbf{Pattern}$ and $\mathbf{Thread}$, and has morphisms including:
  \begin{mathpar}
    \mathbf{cut} \in \mathcal{M}(\mathbf{Fabric};\mathbf{Pattern})

    \mathbf{sew}\in \mathcal{M}(\mathbf{Pattern},\mathbf{Thread};\mathbf{Shirt})

    \mathbf{pack} \in \mathcal{M}(\mathbf{Pants},\mathbf{Shirt};\mathbf{Clothes})
  \end{mathpar}
  Then let $A$ be the derivation:
  \begin{mathpar}
    \inferrule*[right=putr]
    {
      \inferrule*[right=seq]
      {
        \inferrule*[right=var]
        {\text{ }}
        {p : \mathbf{Pants} \vdash p : \mathbf{Pants}}
      }
      {p : \mathbf{Pants} \Vdash [p] : (\lambda,\mathbf{Pants},\lambda)}
      \\
      \inferrule*[right=seq]
      {
        f : \mathbf{Fabric} \vdash \mathbf{cut}(f) : \mathbf{Pattern}
      }
      {f : \mathbf{Fabric} \Vdash [\mathbf{cut}(f)] : (\lambda,\mathbf{Pattern},\lambda)}}
    {p : \mathbf{Pants},f : \mathbf{Fabric} \Vdash \putr(\mathbf{cut}(f),[p]) : (\lambda,\mathbf{Pants},\mathbf{Pattern}^\circ)}
  \end{mathpar}
  and let $B$ be the derivation:
  \begin{mathpar}
    \inferrule*[right=getl]
    {
      \inferrule*[right=seq]
      {
        a : \mathbf{Pattern},t : \mathbf{Thread} \vdash \mathbf{sew}(a,t) : \mathbf{Shirt}
      }
      {a : \mathbf{Pattern}, t : \mathbf{Thread} \Vdash [\mathbf{sew}(a,t)] : (\lambda,\mathbf{Shirt},\lambda)}
    }
    {t : \mathbf{Thread} \Vdash \getl(a.[\mathbf{sew}(a,t)]) : (\mathbf{Pattern}^\circ,\mathbf{Shirt},\lambda)}
  \end{mathpar}
  Then we may derive:
  {\small
  \begin{mathpar}
    \inferrule*[right=let]
    {
      A
      \\
      B
      \\
      \inferrule*[right=seq]
      {x : \mathbf{Pants},y : \mathbf{Shirt} \vdash \mathbf{pack}(x,y) : \mathbf{Clothes}}
      {x : \mathbf{Pants},y : \mathbf{Shirt} \Vdash [\mathbf{pack}(x,y)] : (\lambda,\mathbf{Clothes},\lambda)}
    }
    {p : \mathbf{Pants},f : \mathbf{Fabric} , t : \mathbf{Thread} \Vdash \textsf{let } x,y \asn (\putr(\mathbf{cut}(f),[p]) \mid \getl(a.[\mathbf{sew}(a,t)])) \textsf{ in } [\mathbf{pack}(x,y)] : (\lambda,\mathbf{Clothes},\lambda)}
  \end{mathpar}}
Which rewrites to normal form as in:
\begin{align*}
  &\textsf{let } x,y \asn (\putr(\mathbf{cut}(f),[p]) \mid \getl(a.[\mathbf{sew}(a,t)])) \textsf{ in } [\mathbf{pack}(x,y)]
  \\&\to \textsf{let } x,y \asn ([p] \mid [\mathbf{sew}(a,t)][\mathbf{cut}(f)/a])  \textsf{ in } [\mathbf{pack}(x,y)]
  \\&=  \textsf{let } x,y \asn ([p] \mid [\mathbf{sew}(\mathbf{cut}(f),t)])  \textsf{ in } [\mathbf{pack}(x,y)]
  \\&\to [\mathbf{pack}(x,y)][p,\mathbf{sew}(\mathbf{cut}(f),t)/x,y]
  \\&= [\mathbf{pack}(p,\mathbf{sew}(\mathbf{cut}(f),t))]
\end{align*}
In this way, subterms of a message passing term may exchange resources during evaluation.
\end{example}

\begin{example}
	Say $\mathcal{M}$ has objects including $\mathbf{Person}$, $\mathbf{Money}$, $\mathbf{Beans}$, $\mathbf{Water}$, $\mathbf{Coffee}$, and $\mathbf{Ready}$, and has morphisms including:
	\begin{mathpar}
		\mathbf{brew} \in \mathcal{M}(\mathbf{Money},\mathbf{Beans},\mathbf{Water};\mathbf{Coffee})
		
		\mathbf{drink} \in \mathcal{M}(\mathbf{Person},\mathbf{Coffee};\mathbf{Person})
		
		\mathbf{done} \in \mathcal{M}(;\mathbf{Ready})
	\end{mathpar}
	Let $A$ be the term modeling a person paying for a coffee at a coffee machine, and then drinking it:
	\[
	p : \mathbf{Person} , m : \mathbf{Money} \Vdash \putr(m,\getr(c.[\mathbf{drink}(p,c)])) : (\lambda,\mathbf{Person},\mathbf{Money}^\circ \mathbf{Coffee}^\bullet)
	\]
	and let $B$ be the term modeling the coffee machine, which gets water from the mains:
	\[
	b : \mathbf{Beans} \Vdash \getl(n,\getr(w.\putl(\mathbf{brew}(n,b,w), \mathbf{done}()))) : (\textbf{Money}^\circ \textbf{Coffee}^\bullet ,\mathbf{Ready},\mathbf{Water}^\bullet)
	\]
	Consider the term $\textsf{let } x,y \asn (A \mid B) \textsf{ in } C$ for some well-typed continuation term $C$. Then we have:
	\begin{align*}
		& \textsf{let } x,y \asn (A \mid B) \textsf{ in } C
		\\
		&= \textsf{let } x,y \asn (\putr(m,\getr(c.[\mathbf{drink}(p,c)])) \mid \getl(n,\getr(w.\putl(\mathbf{brew}(n,b,w), \mathbf{done}())))) \textsf{ in } C
		\\
		& \to \textsf{let } x,y \asn (\getr(c.[\mathbf{drink}(p,c)]) \mid \getr(w.\putl(\mathbf{brew}(m,b,w), \mathbf{done}()))) \textsf{ in } C
		\\
		& \to \getr(w.\textsf{let } x,y \asn (\getr(c.[\mathbf{drink}(p,c)]) \mid \putl(\mathbf{brew}(m,b,w), \mathbf{done}())) \textsf{ in } C)
		\\
		& \to \getr(w.\textsf{let } x,y \asn ([\mathbf{drink}(p,\mathbf{brew}(m,b,w))] \mid \mathbf{done}()) \textsf{ in } C)
		\\
		& \to \getr(w. C[\mathbf{drink}(p,\mathbf{brew}(m,b,w)), \mathbf{done}()/x,y])
	\end{align*}
	
\end{example}

We show that our interpretation $\llbracket - \rrbracket$ of terms (Definition~\ref{def:interpretation}) is coherent with respect to the convertibility relation $\stackrel{*}{\leftrightarrow}$ induced by $\to$. It is convenient to do this in two parts. First, we have:
\begin{lemma}\label{lem:substitution-coherent}
  $\llbracket \textsf{let } x_1 \seq x_n \asn ([v_1] \mid \cdots \mid [v_n]) \textsf{ in } t \rrbracket = \llbracket t[v_1 \seq v_n / x_1 \seq x_n]\rrbracket$ whenever this makes sense.
\end{lemma}
\begin{proof}
  We proceed by structural induction on $t$. The cases are as follows:
  \begin{itemize}
  \item If $t = [v]$ then we have:
    \begin{align*}
      & \llbracket \textsf{let } x_1 \seq x_n \asn ([v_1] \mid \cdots \mid [v_n]) \textsf{ in } [v] \rrbracket
        = ( \llbracket [v_1] \rrbracket \mid \cdots \mid \llbracket [v_n] \rrbracket ) \cdot \llbracket [(v)] \rrbracket
        = \left(\corner{(v_1)} \mid \cdots \mid \corner{(v_n)}\right) \cdot \corner{(v)}
      \\&= \left(\corner{(v_1)} \otimes \cdots \otimes \corner{(v_n)}\right) \cdot \corner{(v)}
      = \corner{(v_1\seq v_n)} \cdot \corner{(v)}
      = \corner{(v_1 \seq v_n)(v)}
      = \corner{(v \circ (v_1 \seq v_n))}
      \\&= \llbracket [v[v_1 \seq v_n/x_1 \seq x_n]] \rrbracket      
    \end{align*}
  \item If $t = \textsf{let } y_1 \seq y_m \asn (t_1 \mid \cdots \mid t_m) \textsf{ in } r$ then we have:
    \begin{align*}
      & \llbracket \textsf{let } x_1^1 \seq x_1^{k_1} \seq x_m^1  \seq x_m^{k_m} \asn ([v_1^1] \mid \cdots \mid [v_1^{k_1}] \mid \cdots \mid [v_m^1] \mid \cdots  \mid [v_m^{k_m}]) \textsf{ in } (\textsf{let } y_1 \seq y_m \asn (t_1 \mid \cdots \mid t_m) \textsf{ in } r) \rrbracket
      \\&= (\llbracket [v_1^1] \rrbracket \mid \cdots \mid \llbracket [v_1^{k_1}]\rrbracket \mid \cdots \mid \llbracket [v_m^{k_m}] \rrbracket \mid \cdots \mid \llbracket [v_m^{k_m}] \rrbracket) \cdot \llbracket \textsf{let } y_1 \seq y_m \asn (t_1 \mid \cdots \mid t_m) \textsf{ in } r) \rrbracket
      \\&= (\llbracket [v_1^1] \rrbracket \mid \cdots \mid \llbracket [v_1^{k_1}]\rrbracket \mid \cdots \mid \llbracket [v_m^{k_m}] \rrbracket \mid \cdots \mid \llbracket [v_m^{k_m}] \rrbracket) \cdot (\llbracket t_1 \rrbracket \mid \cdots \mid \llbracket t_m \rrbracket) \cdot \llbracket r \rrbracket
      \\&= (((\llbracket [v_1^1] \rrbracket \mid \cdots \mid \llbracket [v_1^{k_1}] \rrbracket) \cdot \llbracket t_1 \rrbracket) \mid \cdots \mid ((\llbracket [v_m^1] \rrbracket \mid \cdots \mid \llbracket [v_m^{k_m}] \rrbracket) \cdot \llbracket t_m \rrbracket)) \cdot \llbracket r \rrbracket
      \\&= (\llbracket \textsf{let } x_1^1 \seq x_1^{k_1} \asn ([v_1^1] \mid \cdots \mid [v_1^{k_1}]) \textsf{ in } t_1 \rrbracket \mid \cdots \mid \llbracket \textsf{let } x_m^1 \seq x_m^{k_m} \asn ([v_m^1] \mid \cdots \mid [v_m^{k_m}]) \textsf{ in } t_m \rrbracket) \cdot \llbracket r \rrbracket
      \\&= (\llbracket t_1[v_1^1 \seq v_1^{k_1}/x_1^1 \seq x_1^{k_1}] \rrbracket \mid \cdots \mid \llbracket t_m[v_m^1 \seq v_m^{k_m} / x_m^1 \seq x_m^{k_m}] \rrbracket) \cdot \llbracket r \rrbracket
      \\&= \llbracket \textsf{let } y_1\seq y_m \asn (t_1[v_1^1 \seq v_1^{k_1}/x_1^1\seq x_1^{k_1}] \mid \cdots \mid t_m[v_m^1 \seq v_m^{k_m} / x_m^1 \seq x_m^{k_m}]) \textsf{ in } r \rrbracket
    \end{align*}
  \item If $t = \putr(v,t')$ then we have: 
    \begin{align*}
      & \llbracket \textsf{let } x_1 \seq x_n,y_1\seq y_m \asn ([v_1] \mid \cdots \mid [v_n] \mid [u_1] \mid \cdots \mid [u_m]) \textsf{ in } \putr(v,t')\rrbracket
      \\&= (\llbracket [v_1] \rrbracket \mid \cdots \mid \llbracket [v_n] \rrbracket \mid \llbracket [u_1] \rrbracket \mid \cdots \mid \llbracket [u_m] \rrbracket) \cdot \llbracket \putr(v,t') \rrbracket
      \\&= (\llbracket [v_1] \rrbracket \mid \cdots \mid \llbracket [v_n] \rrbracket \mid \llbracket [u_1] \rrbracket \mid \cdots \mid \llbracket [u_m] \rrbracket) \cdot (1 \mid (\corner{(v)} \cdot A\llcorner)) \cdot \llbracket t' \rrbracket
      \\&= ((\llbracket [v_1] \rrbracket \mid \cdots \llbracket [v_n] \rrbracket) \cdot \llbracket t' \rrbracket) \mid ((\llbracket [u_1] \rrbracket \mid \cdots \mid \llbracket [u_m] \rrbracket) \cdot \corner{(v)} \cdot A\llcorner)
      \\&= \llbracket t'[v_1 \seq v_n/x_1 \seq x_n] \rrbracket \mid (\corner{( v[u_1\seq u_m / y_1 \seq y_m] )} \cdot A\llcorner)
      \\&= (1 \mid (\corner{( v[u_1\seq u_m / y_1 \seq y_m] )} \cdot A\llcorner)) \mid \llbracket t'[v_1 \seq v_n / x_1 \seq x_n] \rrbracket 
      \\&= \llbracket \putr(v[u_1\seq u_m/y_1\seq y_m],t'[v_1 \seq v_n/x_1 \seq x_n]) \rrbracket
      \\&= \llbracket \putr(v,t')[v_1\seq v_n,u_1 \seq u_m/x_1 \seq x_n,y_1 \seq y_m] \rrbracket
    \end{align*}
  \item If $t = \putl(v,t')$, a similar argument to the one above suffices.
  \item If $t = \getr(x.t')$, then we have:
    \begin{align*}
      & \llbracket \textsf{let } x_1 \seq x_n \asn ([v_1] \mid \cdots \mid [v_n]) \textsf{ in } \getr(x.t') \rrbracket
        = (\llbracket [v_1] \rrbracket \mid \cdots \mid \llbracket [v_n] \rrbracket ) \cdot \llbracket \getr(x.t') \rrbracket
      \\&= (\llbracket [v_1] \rrbracket \mid \cdots \mid \llbracket [v_n] \rrbracket ) \cdot (1 \mid A\ulcorner) \cdot \llbracket t' \rrbracket
      = \corner{(v_1\seq v_n)} \cdot (1 \mid A\ulcorner) \cdot \llbracket t' \rrbracket
      \\&= (1 \mid A\ulcorner) \cdot \corner{(v_1 \seq v_n,1_A)} \cdot \llbracket t' \rrbracket
      = (1 \mid A\ulcorner) \cdot \llbracket \textsf{let } x_1 \seq x_n,x \asn ([v_1] \mid \cdots [v_n] \mid [x]) \textsf{ in } t' \rrbracket
      \\&= (1 \mid A\ulcorner) \cdot \llbracket t'[v_1\seq v_n,x / x_1 \seq x_n,x] \rrbracket
      = \llbracket \getr(x.t'[v_1 \seq v_n,x / x_1 \seq x_n,x]) \rrbracket
      \\&= \llbracket \getr(x.t')[v_1 \seq v_n / x_1 \seq x_n]\rrbracket
    \end{align*}
  \item If $t = \getl(x.t')$, a similar argument to the one above suffices.
  \end{itemize}
\end{proof}

And then using Lemma~\ref{lem:substitution-coherent} it is straightforward to obtain:
\begin{lemma}\label{lem:interpretation-coherent}
  If $t \stackrel{*}{\leftrightarrow} t'$ then $\llbracket t \rrbracket = \llbracket t' \rrbracket$.
\end{lemma}
\begin{proof}
  It suffices to show that $\llbracket s \rrbracket = \llbracket s' \rrbracket$ for each of our generating rewrites \textbf{[RN]} $s \to s'$. Lemma~\ref{lem:substitution-coherent} shows that this is the case for \textbf{[R0]}. Moving on to the rest of the rewrite rules, for \textbf{[R1]} we have:
  \[ \llbracket \textsf{let } x_1 \seq x_n \asn (t_1 \mid \cdots \mid \putr(v,t) \mid \getl(y.s) \mid \cdots t_n) \textsf{ in } r \rrbracket = \llbracket \textsf{let } x_1 \seq x_n \asn (t_1 \mid \cdots \mid t \mid s[v/y] \mid \cdots \mid t_n) \textsf{ in } r \rrbracket\]
  as in:
  \begin{mathpar}
    \includegraphics[height=2cm,align=c]{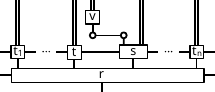} 
    \hspace{0.3cm}
    =
    \hspace{0.3cm}
    \includegraphics[height=2cm,align=c]{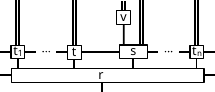} 
  \end{mathpar}
  The case for \textbf{[R2]} is similar. For \textbf{[R3]}, both $\llbracket \textsf{let } x_1 \seq x_n \asn (\putl(v,t) \mid \cdots \mid t_n) \textsf{ in } r \rrbracket$ and $\llbracket \putl(v,\textsf{let } x_1 \seq x_n \asn (t \mid \cdots \mid t_n) \textsf{ in } r) \rrbracket$ are equal to:
  \[
    \includegraphics[height=2cm,align=c]{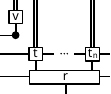} 
  \]
  For \textbf{[R4]}, both $\llbracket \textsf{let } x_1 \seq x_n \asn (\getl(x.t) \mid \cdots \mid t_n) \textsf{ in } r \rrbracket$ and $\llbracket \getl(x.\textsf{let } x_1 \seq x_n \asn (t \mid \cdots \mid t_n) \textsf{ in } r) \rrbracket$ are equal to:
  \[
    \includegraphics[height=1.634cm,align=c]{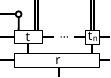} 
  \]
  The case for \textbf{[R5]} is similar to \textbf{[R3]}, and the case for \textbf{[R6]} is similar to \textbf{[R4]}. For \textbf{[R7]}, both $\llbracket \putr(v,\putl(w,t)) \rrbracket$ and $\llbracket \putl(w,\putr(v,t)) \rrbracket$ are equal to:
  \[
    \includegraphics[height=1.5cm,align=c]{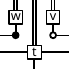} 
  \]
  The cases for \textbf{[R8]}, \textbf{[R9]}, and \textbf{[R10]} are similar, and the claim follows.
\end{proof}

Lemma~\ref{lem:interpretation-coherent} will be important later on, when we show that $\llbracket - \rrbracket$ gives a functor of virtual double categories. In particular, our terms $T(\mathcal{M})$ will form a virtual double category modulo $\stackrel{*}{\leftrightarrow}$, and Lemma~\ref{lem:interpretation-coherent} is required for $\llbracket - \rrbracket$ to be a well-defined function from $\stackrel{*}{\leftrightarrow}$-equivalence classes to cells of the free cornering.

\subsection{Termination}
We proceed to show that $\to$ is terminating and confluent. Termination is straightforward: we assign each term a size and show that each of our generating rewrites decreases this size. The size is given as in:
\begin{definition}
  We assign to each term $t$ a \emph{size} $\#(t) \in \mathbb{N}$ as follows:
  \[
    \#(t) = \begin{cases}
            2 & \text{if } t = [v]\\
            2 \cdot \#(t_1) \cdots \#(t_n) \cdot \#(s) & \text{if } t = \textsf{let } x_1,\ldots,x_n \asn (t_1\mid \cdots \mid t_n) \textsf{ in } s\\
              2 + (2 \cdot \#(s)) & \text{if } t = \putl(v,s) \textsf{ or } t = \getl(x.s)\\
              1 + (2 \cdot \#(s)) & \text{if } t = \putr(v,s) \textsf{ or } t = \getr(x.s)
          \end{cases}
  \]
\end{definition}
Next, we require an auxiliary lemma concerning value substitution:
\begin{lemma}\label{lem:subst-size}
  $\#(t) = \#(t[v_1\seq v_n / x_1 \seq x_n])$
\end{lemma}
\begin{proof}
  By structural induction on $t$. The cases are as follows:
  \begin{itemize}
  \item If $t = [v]$ then $\#([v]) = 2 = \#([v[v_1\seq v_n/x_1 \seq x_n]]) = \#([v][v_1 \seq v_n / x_1 \seq x_n])$.
  \item If $t = \textsf{let } y_1 \seq y_m \asn (t_1 \mid \cdots \mid t_n) \textsf{ in } r$ then we have:
    \begin{align*}
      & \#((\textsf{let } y_1 \seq y_m \asn (t_1 \mid \cdots \mid t_n) \textsf{ in } r)[\overline{v_1}\seq \overline{v_n} / \overline{x_1} \seq \overline{x_n}])
      \\&=  \#(\textsf{let } y_1 \seq y_m \asn (t_1[\overline{v_1}/\overline{x_1}] \mid \cdots \mid t_n[\overline{v_n}/\overline{x_n}]) \textsf{ in } r)
      = 2 \cdot \#(t_1[\overline{v_1}/\overline{x_1}]) \cdots \#(t_n[\overline{v_n}/\overline{x_n}]) \cdot \#(r)
      \\&=  2 \cdot \#(t_1) \cdots \#(t_n) \cdot \#(r)
      = \#(\textsf{let } y_1 \seq y_m \asn (t_1 \mid \cdots \mid t_n) \textsf{ in } r)
    \end{align*}
  \item If $t = \putr(v,t')$ then we have:
    \begin{align*}
      & \#(\putr(v,t')[\overline{v_1},\overline{v_2}/\overline{x_1},\overline{x_2}])
        = \#(\putr(v[\overline{v_2}/\overline{x_2}],t'[\overline{v_1}/\overline{x_1}]))
        \\&= 1 + (2 \cdot \#(t'[\overline{v_1}/\overline{x_1}]))
        = 1 + (2 \cdot \#(t'))
        = \#(\putr(v,t'))
    \end{align*}
  \item The case where $t = \putl(v,t')$ is similar.
  \item If $t = \getl(y.t')$ then we have:
    \begin{align*}
      & \#(\getl(y.t')[\overline{v}/\overline{x}])
        = \#(\getl(y.t'[\overline{v}/\overline{x}]))
        = 2 + (2 \cdot \#(t'[\overline{v}/\overline{x}]))
        = 2 + (2 \cdot \#(t'))
        = \#(\getl(y.t'))
    \end{align*}
  \item The case where $t = \getr(y.t')$ is similar.
  \end{itemize}
\end{proof}

From here, we easily obtain that reduction decreases the size of the term involved:
\begin{lemma}\label{lem:decreasing}
  If $t \to t'$ then $\#(t) > \#(t')$.
\end{lemma}
\begin{proof}
  For \textbf{[R0]} Lemma~\ref{lem:subst-size} gives:
  \begin{align*}
    & \#(\textsf{let } x_1 \seq x_n \asn ([v_1]\mid \cdots \mid [v_n]) \textsf{ in } t)
      = 2 \cdot \#([v_1]) \cdots \#([v_n]) \cdot \#(t)
      = 2^{n+1} \cdot \#(t)
      \\&> \#(t)
      = \#(t[v_1 \seq v_n/x_1 \seq x_n])
  \end{align*}
  For \textbf{[R1]} we have:
  \begin{align*}
    & \#(\textsf{let } x_1 \seq x_n \asn (\overline{a} \mid \putr(v,t) \mid \getl(y.s) \mid \overline{b}) \textsf{ in } r)
      = 2 \cdot \#(\overline{a}) \cdot \#(\putr(v,t)) \cdot \#(\getl(y.s)) \cdot \#(\overline{b}) \cdot \#(r)
    \\&= 2 \cdot \#(\overline{a}) \cdot (1 + (2 \cdot \#(t))) \cdot (2 + (2 \cdot \#(s))) \cdot \#(\overline{b}) \cdot \#(r)
    > 2 \cdot \#(\overline{a}) \cdot \#(t) \cdot \#(s) \cdot \#(\overline{b}) \cdot \#(r)
    \\& = 2 \cdot \#(\overline{a}) \cdot \#(t) \cdot \#(s[v/y]) \cdot \#(\overline{b}) \cdot \#(r)
    = \#(\textsf{let } x_1 \seq x_n \asn (\overline{a} \mid t \mid s[v/y] \mid \overline{b}) \textsf{ in } r)
  \end{align*}
  The case concerning \textbf{[R2]} is similar. For \textbf{[R3]} we have:
  \begin{align*}
    & \#(\textsf{let } x_1 \seq x_n \asn (\putl(v,t) \mid \overline{a}) \textsf{ in } r)
      = 2 \cdot \#(\putl(v,t)) \cdot \#(\overline{a}) \cdot \#(r)
    \\& = 2 \cdot (2+(2\cdot \#(t))) \cdot \#(\overline{a}) \cdot \#(r)
    > 2 + (4 \cdot \#(t) \cdot \#(\overline{a}) \cdot \#(r))
    \\&= 2 + (2 \cdot \#(\textsf{let } x_1 \seq x_n \asn (t \mid \overline{a}) \textsf{ in } r))
    = \#(\putl(v,\textsf{let } x_1 \seq x_n \asn (t \mid \overline{a}) \textsf{ in } r)
  \end{align*}
  The cases concerning \textbf{[R4]}, \textbf{[R5]}, and \textbf{[R6]} are similar. For \textbf{[R7]} we have:
  \begin{align*}
    & \#(\putr(v,\putl(w,t)))
      = 1 + (2 \cdot \#(\putl(w,t)))
      = 1 + (2 \cdot (2 + (2 \cdot \#(t))))
      = 1 + 4 + (4 \cdot \#(t))
      \\&=  5 + (4 \cdot \#(t))
    > 4 + (4 \cdot \#(t))
    = 2 + 2 + (4 \cdot \#(t))
    = 2 + (2 \cdot (1 + (2 \cdot \#(t))))
    \\&= 2 + (2 \cdot \#(\putr(v,t)))
    = \#(\putl(w,\putr(v,t)))
  \end{align*}
  The cases concerning \textbf{[R8]}, \textbf{[R9]}, and \textbf{[R10]} are similar. 
  
  Note finally that decreasing the size of a subterm reduces the size of the whole term, and hence when reducing a subterm the size also decreases. The claim follows.
\end{proof}

We may now record, since $\to$ reduces the size and this cannot be done infinitely many times, that:
\begin{corollary}\label{lem:termination}
  The rewrite relation $\to$ is terminating
\end{corollary}

\subsection{Confluence}

Our next goal will be to show that $\to$ is confluent. The following lemma is helpful:
\begin{lemma}\label{lem:let-reduce}
  Any let-binding reduces.
\end{lemma}
\begin{proof} 
  We proceed by induction on term size. The base case is when the size of the let-binding is $0$, but no term has size zero, so the claim holds vacuously. For the inductive case, suppose $t = \textsf{let } x_1 \seq x_n \asn (t_1 \mid \cdots \mid t_n) \textsf{ in } r$, and suppose that for any let-binding $t'$ such that $\#(t') < \#(t)$, $t'$ reduces. The cases are:
  \begin{itemize}
  \item If any $t_1 \seq t_n$ is a let-binding then by the inductive hypothesis that term reduces, and so does $t$. We may therefore assume that no $t_i$ is a let-binding in the remaining cases.
  \item If $t_i = [v_i]$ for all $1 \leq i \leq n$ then $t$ reduces via \textbf{[R0]}.
  \item If at least one of the $t_i$ is left-facing, let $i$ be the smallest index $1 \leq i \leq n$ for which this is so. If $i = 1$, then $t$ reduces via \textbf{[R3]} or \textbf{[R4]}. If $i > 1$, then $t_{i-1}$ cannot be left-facing since $i - 1 < i$, cannot be neutral due to typing constraints, and cannot be a let-binding by assumption. It follows that $t_{i-1}$ must be right-facing. Typing constraints ensure that if $t_i = \putl(v.t_i')$ then $t_{i-1} = \getr(x.t_{i-1}')$ so that $t$ reduces via \textbf{[R2]}, and that if $t_i = \getl(x.t_i')$ then $t_{i-1} = \putr(v,t_{i-1}')$ so that $t$ reduces via \textbf{[R1]}.
  \item Finally, if none of $t_1 \seq t_n$ are left-facing, then we must have that at least one of the $t_i$ is right-facing. Let $i$ be the largest index $1 \leq i \leq n$ for which this is so. If $i = n$ then $t$ reduces via \textbf{[R5]} or \textbf{[R6]}. The case where $i < n$ is impossible: $t_{i+1}$ cannot be neutral due to typing constraints, cannot be right-facing since $i+1 > i$, and cannot be left-facing or a let-binding by assumption. There are no other terms. 
  \end{itemize}
\end{proof}

We note that Lemma~\ref{lem:let-reduce} may be viewed as a kind of cut-elimination theorem for the sequent calculus that presents $T(\mathcal{M})$ (Definition~\ref{def:terms}). From the perspective of our interpretation of the calculus in terms of interacting processes, we can think of this in terms of ``deadlock freedom'': no term in normal form may contain a let-binding, which means that any internally resolvable interaction in a term of $T(\mathcal{M})$ is in fact resolved by reducing the term in question to normal form.

Our proof of confluence relies on a bit of technical machinery concerning terms in context. We define:
\begin{definition}
  For each $U,V,W \in \ex{\mathcal{M}_0}$, each $B \in \mathcal{M}_0$, and each $\Gamma,\Delta \in \mathcal{M}_0^*$ the \emph{term contexts}
  \[ \mathcal{L} : \floatwiggle{U}{\blacksquare}{\Delta}{\blacksquare} \rightsquigarrow \floatwiggle{V}{\Gamma,\blacksquare}{B}{\blacksquare W} \]
  are constructed according to the following inference rules:
\begin{mathpar}
  \inferrule{(\Gamma_i \Vdash a_i : (U_{i-1},A_i,U_i))_{i = 1}^k \\ x_1 {:} A_1 \seq x_n {:} A_n \Vdash r : (V,B,W) \\ k < n}{\textsf{let } x_1 \seq x_n \asn (a_1 \mid \cdots \mid a_k \mid \blacksquare) \textsf{ in } r : \floatwiggle{U_k}{\blacksquare}{A_{n-k},\ldots, A_n}{\blacksquare} \rightsquigarrow \floatwiggle{U_0 V}{\Gamma_1 ,\ldots, \Gamma_k,\blacksquare}{B_{}}{\blacksquare W}}

  \inferrule{\Sigma \vdash_{\mathcal{M}}\, v : A \\ \mathcal{L} : \floatwiggle{U}{\blacksquare}{\Delta_{}}{\blacksquare} \rightsquigarrow \floatwiggle{V}{\Gamma,\blacksquare}{B_{}}{\blacksquare W}}{\putl(v,\mathcal{L}) : \floatwiggle{U}{\blacksquare}{\Delta_{}}{\blacksquare} \rightsquigarrow \floatwiggle{A^\bullet V}{\Sigma,\Gamma,\blacksquare}{B_{}}{\blacksquare W}}

  \inferrule{\mathcal{L} : \floatwiggle{U}{\blacksquare}{\Delta_{}}{\blacksquare} \rightsquigarrow \floatwiggle{V}{A,\Gamma,\blacksquare}{B_{}}{\blacksquare W}}{\getl(x.\mathcal{L}) : \floatwiggle{U}{\blacksquare}{\Delta_{}}{\blacksquare} \rightsquigarrow \floatwiggle{A^\circ V}{\Gamma,\blacksquare}{B_{}}{\blacksquare W}}
\end{mathpar}
Contexts $\mathcal{L}$ may be pictured as in:
\begin{mathpar}
  \mathcal{L}
  \hspace{0.3cm}
  \leftrightsquigarrow
  \hspace{0.3cm}
  \includegraphics[height=1.25cm,align=c]{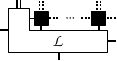} 
\end{mathpar}
in which case the individual inference rules for constructing term contexts are pictured as in:
\begin{mathpar}
  \putl(v,\mathcal{L})
  \hspace{0.3cm}
  \leftrightsquigarrow
  \hspace{0.3cm}
  \includegraphics[height=2cm,align=c]{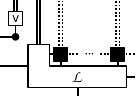} 

  \getl(x.\mathcal{L})
  \hspace{0.3cm}
  \leftrightsquigarrow
  \hspace{0.3cm}
  \includegraphics[height=1.8cm,align=c]{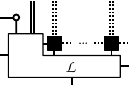} 

  \textsf{let } x_1\seq x_n \asn (a_1 \mid \cdots \mid a_k \mid \blacksquare) \textsf{ in } r
  \hspace{0.3cm}
  \leftrightsquigarrow
  \hspace{0.3cm}
  \includegraphics[height=1.25cm,align=c]{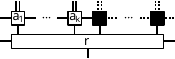} 
\end{mathpar}

We obtain terms from term contexts by supplying a sequence of terms of the appropriate type. Specifically, if $\mathcal{L} : \wiggle{U}{\blacksquare}{A_1,\ldots,A_k}{\blacksquare} \rightsquigarrow \wiggle{V}{\Gamma,\blacksquare}{B}{\blacksquare W}$ is a term context, then whenever $(\Gamma_i \Vdash t_i : (U_{i-1},A_i,U_i))_{i =1}^k$ with $U_0 = U$, we have $\Gamma,\Gamma_1 \seq \Gamma_k \Vdash \mathcal{L}[t_1 \seq t_k] : (V,B,U_k W)$ with $\mathcal{L}[-]$ defined as in:
\begin{mathpar}
  \mathcal{L}[t_1\seq t_k] =
  \begin{cases}
    \textsf{let } x_1 \seq x_n \asn (a_1 \mid \cdots \mid a_m \mid t_1 \mid \cdots \mid t_k) \textsf{ in } r &\text{if } \mathcal{L} = \textsf{let } x_1 \seq x_n \asn (a_1\mid \cdots \mid a_m \mid \blacksquare) \textsf{ in } r\\
    \putl(v,\mathcal{L}'[t_1 \seq t_k]) &\text{if } \mathcal{L} = \putl(v,\mathcal{L}')\\
    \getl(x.\mathcal{L}'[t_1 \seq t_k]) &\text{if } \mathcal{L} = \getl(x.\mathcal{L}')
  \end{cases}
\end{mathpar}
\end{definition}

We write $s(\mathcal{L})$ to indicate the let-bound sequence of terms in $\mathcal{L}$. Explicitly:
\[
  s(\mathcal{L}) =
  \begin{cases}
    a_1 \seq a_k &\text{if } \mathcal{L} = \textsf{let } x_1 \seq x_n \asn (a_1 \mid \cdots \mid a_k \mid \blacksquare) \textsf{ in } r\\
    s(\mathcal{L}') &\text{if } \mathcal{L} = \putl(v,\mathcal{L}') \text{ or } \mathcal{L} = \getl(x.\mathcal{L}')\\
  \end{cases}
\]
Note that $s(\mathcal{L})$ may be empty. In case it is not, we call $\mathcal{L}$ \emph{nonempty}. A nonempty term context $\mathcal{L} : \wiggle{U}{\blacksquare}{\Delta}{\blacksquare} \rightsquigarrow \wiggle{V}{\Gamma,\blacksquare}{B}{\blacksquare W}$ is said to be \emph{effectful} in case $U \neq \lambda$. An effectful term context $\mathcal{L}$ is said to be \emph{active} in case $s(\mathcal{L}) = a_1 \seq a_k$ where $a_k$ is right-facing. Note further that for any nonempty $\mathcal{L} : \wiggle{U}{\blacksquare}{\Delta}{\blacksquare} \rightsquigarrow \wiggle{V}{\Gamma,\blacksquare}{B}{\blacksquare W}$ with $s(\mathcal{L}) = a_1\seq a_k,a$ where $a : \cells{U'}{\Sigma}{C}{U}$ there exists another term context $\mathcal{L}' : \wiggle{U'}{\blacksquare}{C,\Delta}{\blacksquare} \rightsquigarrow \wiggle{V}{\Gamma,\Sigma,\blacksquare}{B}{\blacksquare W}$ such that $\mathcal{L}[t_1 \seq t_n] = \mathcal{L}'[a,t_1\seq t_n]$. In particular this means that for any active $\mathcal{L}$ with $s(\mathcal{L}) = a_1 \seq a_k,\putr(v,t)$ there exists $\mathcal{L}'$ such that $\mathcal{L}[t_1\seq t_n] = \mathcal{L}'[\putr(v,t),t_1\seq t_n]$, and similarly if $s(\mathcal{L}) = a_1\seq a_k ,\getr(x.t)$ there exists $\mathcal{L}'$ such that $\mathcal{L}[t_1\seq t_n] = \mathcal{L}'[\getr(x.t),t_1 \seq t_n]$. We say that $\mathcal{L} \to \mathcal{L}'$ in case for all suitable $t_1,\ldots,t_k$ we have $\mathcal{L}[t_1 \seq t_k] \to \mathcal{L}'[t_1 \seq t_k]$. Moreover, we define $\#(\mathcal{L}) \in \mathbb{N}$ as follows:
\[
  \#(\mathcal{L}) = \begin{cases}
		2 \cdot\#(a_1) \cdots \#(a_n) \cdot \#(r) & \text{if } \mathcal{L} = \textsf{let } x_1 \seq x_n \asn (a_1 \mid \cdots \mid a_k \mid \blacksquare) \textsf{ in } r\\
		2 + (2 \cdot \#(\mathcal{L}')) & \text{if } \mathcal{L} = \putl(v,s) \textsf{ or } \mathcal{L} = \getl(x.\mathcal{L}')
	\end{cases}
\]
Note that if $\mathcal{L} \to \mathcal{L}'$ then $\#(\mathcal{L}) > \#(\mathcal{L}')$, via a straightforward extension of the proof of Lemma \ref{lem:termination}.

We may now state the main lemma concerning term contexts:
\begin{lemma}\label{lem:effectful-active}
  Let $\mathcal{L}$ be an effectful term context, then there is an active term context $\mathcal{L}'$ such that $\mathcal{L} \stackrel{*}{\to} \mathcal{L}'$.
\end{lemma}
\begin{proof}
  We proceed by induction on term context size. The base case is when the size is zero, but clearly no term context has size zero so the claim holds vacuously. For the inductive case, let $\mathcal{L}$ be a term context, and suppose the claim holds for each term context $\mathcal{L}'$ with $\#(\mathcal{L}') < \#(\mathcal{L})$. Let $s(\mathcal{L}) = a_1 \seq a_k$. We consider a number of cases. If $\mathcal{L} = \putl(v,\mathcal{L}')$ then $\mathcal{L}'$ is also an effectful term context and we have $\#(\mathcal{L}') < \#(\mathcal{L})$, so the inductive hypothesis gives an active term context $\mathcal{L}''$ such that $\mathcal{L}' \stackrel{*}{\to} \mathcal{L}''$. But then $\putl(v,\mathcal{L}'')$ is also an active term context and we have $\mathcal{L} = \putl(v,\mathcal{L}') \stackrel{*}{\to} \putl(v,\mathcal{L}'')$ and the claim holds. The case in which $\mathcal{L} = \getl(x.\mathcal{L}')$ is similar. Only the case where $\mathcal{L} = \textsf{let } x_1 \seq x_n \asn (a_1 \mid \cdots \mid a_k \mid \blacksquare) \textsf{ in } r$ remains. We consider a number of subcases:
  \begin{itemize}
  \item If any of the $a_i$ is a let-binding then Lemma~\ref{lem:let-reduce} tells us that we have $a_i \to a_i'$ for some $a_i'$. Then $\mathcal{L} \to \mathcal{L'}$ where $\mathcal{L'}$ is the term context obtained from $\mathcal{L}$ by replacing $a_i$ with $a_i'$. Lemma~\ref{lem:decreasing} gives that $\#(a_i) > \#(a_i')$, which means $\#(\mathcal{L}) > \#(\mathcal{L}')$, so our inductive hypothesis gives $\mathcal{L} \to \mathcal{L}' \stackrel{*}{\to} \mathcal{L}''$ for some active $\mathcal{L}''$. We assume that no $a_i$ is a let-binding in the remaining cases. 
    
  \item If at least one of the $a_i$ is left-facing, let $i$ be the smallest index $1 \leq i \leq n$ for which this is so. If $i = 1$ then $\mathcal{L}$ reduces as follows: Say $a_1 = \putl(v,a)$, then we have:
    \begin{align*}
      & \mathcal{L}[t_1\seq t_m]
        = \textsf{let } x_1 \seq x_n \asn (\putl(v,a) \mid \cdots \mid a_k \mid t_1 \mid \cdots \mid t_m)\textsf{ in } r
      \\& \to \putl(v, \textsf{let } x_1 \seq x_n \asn (a \mid \cdots \mid a_k \mid t_1 \mid \cdots \mid t_m) \textsf{ in } r)
    \end{align*}
    Notice that the effectful term context $\mathcal{L}' = \textsf{let } x_1 \seq x_n \asn (a \mid \cdots \mid a_k \mid \blacksquare) \textsf{ in } r$ has $\#(\mathcal{L}') < \#(\mathcal{L})$ and so the inductive hypothesis gives an active $\mathcal{L}''$ such that $\mathcal{L}' \stackrel{*}{\to} \mathcal{L}''$. Now $\putl(v,\mathcal{L}'')$ is also active and we have $\mathcal{L} \to \putl(v,\mathcal{L}') \to \putl(v,\mathcal{L}'')$, so the claim holds. The case where $a_1 = \getl(x.a)$ is similar. If $i > 1$ then $a_{i-1}$ must be right-facing: It cannot be left-facing since $i-1 < i$, cannot be neutral due to typing constraints, and cannot be a let-binding by assumption. In particular if $a_i = \putl(v,a)$ then typing constraints ensure that $a_{i-1} = \getr(x.b)$ and we have:
    \begin{align*}
      & \mathcal{L}[t_1 \seq t_m]
        = \textsf{let } x_1 \seq x_n \asn (a_1 \mid \cdots \mid \getr(x.b) \mid \putl(v,a) \mid \cdots \mid a_k \mid t_1 \mid \cdots \mid t_m) \textsf{ in } r
        \\& \to \textsf{let } x_1 \seq x_n \asn (a_1 \mid \cdots \mid b[v/x] \mid a \mid \cdots \mid a_k \mid t_1 \mid \cdots \mid t_m) \textsf{ in  } r
    \end{align*}
    Now $\mathcal{L}' = \textsf{let } x_1 \seq x_n \asn (a_1 \mid \cdots \mid b[v/x] \mid a \mid \cdots \mid a_k \mid \blacksquare) \textsf{ in  } r$ is effectful and certainly $\#(\mathcal{L}') < \#(\mathcal{L})$ so the inductive hypothesis gives active $\mathcal{L}''$ such that $\mathcal{L}' \stackrel{*}{\to} \mathcal{L}''$. Then we have $\mathcal{L} \to \mathcal{L}' \stackrel{*}{\to} \mathcal{L}''$, so the claim holds. The case where $a_i = \getl(x.a)$ is similar.
  \item If no $a_i$ is left-facing, let $i$ be the greatest index for which $a_i$ is right-facing. If $i = k$ then $\mathcal{L}$ is already active. The case where $i < k$ is impossible: $a_{i+1}$ cannot be right-facing because $i+1 > i$, cannot be neutral due to typing constraints, and cannot be left-facing or a let-binding by assumption. 
  \end{itemize}
  Note that since $\mathcal{L}$ is effectful $a_k$ cannot be neutral, so we need not consider the case in which the $a_i$ are all neutral. The claim follows.
\end{proof}

We require one further technical lemma:
\begin{lemma}\label{lem:boundary-multipop}
  Let $\mathcal{L}$ be a term context. Then we have:
  \begin{enumerate}
  \item\label{sublem:boundary-multipop-first} $\mathcal{L}[b_1\seq b_n , \putr(v,t)]$ and $\putr(v,\mathcal{L}[b_1 \seq b_n , t])$ are joinable.
  \item\label{sublem:boundary-multipop-second} $\mathcal{L}[b_1\seq b_n ,\getr(x.t)]$ and $\getr(x.\mathcal{L}[b_1 \seq b_n,t])$ are joinable.
  \end{enumerate}
  whenever these expressions make sense. Here when we say two terms $t$ and $t'$ of $T(\mathcal{M})$ are \emph{joinable} we mean that there exists some $h$ such that $t \stackrel{*}{\to} h \stackrel{*}{\from} t'$.
\end{lemma}
\begin{proof}
  The proofs of (\ref{sublem:boundary-multipop-first}) and (\ref{sublem:boundary-multipop-second}) are similar. We give the proof of (\ref{sublem:boundary-multipop-first}). We proceed by induction on the form of $\mathcal{L}$. The cases are:
  \begin{itemize}
  \item If $\mathcal{L} = \textsf{let } x_1 \seq x_n \asn (\overline{a} \mid \blacksquare) \textsf{ in } r$ then we have:
    \begin{align*}
      & \mathcal{L}[\overline{b} , \putr(v,t)]
        = \textsf{let } x_1 \seq x_n \asn (\overline{a} \mid \overline{b} \mid \putr(v,t)) \textsf{ in } r
        \\& \to \putr(v,\textsf{let } x_1 \seq x_n \asn (\overline{a} \mid \overline{b} \mid t) \textsf{ in } r)
        = \putr(v,\mathcal{L}[\overline{b} , t])
    \end{align*}
  \item If $\mathcal{L} = \putl(u,\mathcal{L}')$ and the claim holds for $\mathcal{L}'$ then we have:
    \begin{align*}
      & \mathcal{L}[\overline{b} , \putr(v,t)]
        = \putl(u,\mathcal{L}'[\overline{b} , \putr(v,t)])
        \stackrel{*}{\to} h
        \stackrel{*}{\from} \putl(u,\putr(v,\mathcal{L}'[\overline{b} , t]))
        \\& \from \putr(v,\putl(u,\mathcal{L}'[\overline{b} , t]))
        = \putr(v,\mathcal{L}[\overline{b} , t])
    \end{align*}
  \item If $\mathcal{L} = \getl(x.\mathcal{L}')$ and the claim holds for $\mathcal{L}'$ then we have:
    \begin{align*}
      & \mathcal{L}[\overline{b} , \putr(v,t)]
        = \getl(x.\mathcal{L}'[\overline{b} , \putr(v,t)])
        \stackrel{*}{\to} h
        \stackrel{*}{\from} \getl(x.\putr(v,\mathcal{L}'[\overline{b} , t]))
      \\& \from \putr(v,\getl(x.\mathcal{L}'[\overline{b} , t]))
      = \putr(v,\mathcal{L}[\overline{b} , t])
    \end{align*}
  \end{itemize}
\end{proof}

We know that $\to$ is terminating, so to show that it is confluent it suffices to show that it is locally confluent. This is proven by analysis of the possible critical pairs, which we show are all joinable. The term contexts play an important technical role in joining certain critical pairs. 
\begin{lemma}\label{lem:local-confluence}
  $\to$ is locally confluent.
\end{lemma}
\begin{proof}
  There are twelve classes of critical pairs. We number them according to which rewrite rules they arise from, so that for example critical pairs of the form (2-5) arise from the interaction of rewrite rule \textbf{[R2]} and \textbf{[R5]}. We arrange the twelve classes of critical pair into three groups. The first group is:
  \begin{itemize}
  \item[] (1-7): \vspace{-2em}
    \begin{align*}
      & \textsf{let } x_1 \seq x_n \asn (\overline{a} \mid \putl(u,t) \mid s[v/x] \mid \overline{b}) \textsf{ in } r 
      \\& \overset{\text{{\bf [R1]}}}{\from} \textsf{let } x_1 \seq x_n \asn (\overline{a} \mid \putr(v,\putl(u,t)) \mid \getl(x.s) \mid \overline{b}) \textsf{ in } r
      \\& \overset{\text{\bf [R7]}}{\to} \textsf{let } x_1 \seq x_n \asn (\overline{a} \mid \putl(u,\putr(v,t)) \mid \getl(x.s) \mid \overline{b}) \textsf{ in } r 
    \end{align*}
  \item[] (1-8): \vspace{-2em}
    \begin{align*}
     & \textsf{let } x_1 \seq x_n \asn (\overline{a} \mid \getl(y.t) \mid s[v/x] \mid \overline{b}) \textsf{ in } r 
     \\& \overset{\text{\bf[R1]}}{\from} \textsf{let } x_1 \seq x_n \asn (\overline{a} \mid \putr(v,\getl(y.t)) \mid \getl(x.s) \mid \overline{b}) \textsf{ in }r
     \\&\overset{\text{\bf[R8]}}{\to} \textsf{let } x_1 \seq x_n \asn (\overline{a} \mid \getl(y.\putr(v,t)) \mid \getl(x.s) \mid \overline{b}) \textsf{ in } r 
    \end{align*}
  \item[] (2-9): \vspace{-2em}
    \begin{align*}
      &  \textsf{let } x_1 \seq x_n \asn (\overline{a} \mid \putl(v,t)[u/x] \mid s \mid \overline{b}) \textsf{ in } r 
      \\& \overset{\text{\bf[R2]}}{\from} \textsf{let } x_1 \seq x_n \asn (\overline{a} \mid \getr(x.\putl(v,t)) \mid \putl(u,s) \mid \overline{b}) \textsf{ in } r 
      \\& \overset{\text{\bf[R9]}}{\to} \textsf{let } x_1 \seq x_n \asn (\overline{a} \mid \putl(v,\getr(x.t)) \mid \putl(u,x) \mid \overline{b}) \textsf{ in } r 
    \end{align*}
  \item[] (2-10): \vspace{-2em}
    \begin{align*}
      & \textsf{let } x_1 \seq x_n \asn (\overline{a} \mid \getl(y.t)[u/x] \mid s \mid \overline{b}) \textsf{ in } r 
      \\& \overset{\text{\bf[R2]}}{\from} \textsf{let } x_1 \seq x_n \asn (\overline{a} \mid \getr(x.\getl(y.t)) \mid \putl(u,s) \mid \overline{b}) \textsf{ in } r
      \\& \overset{\text{\bf[R10]}}{\to}\textsf{let } x_1 \seq x_n \asn (\overline{a} \mid \getl(y.\getr(x.t)) \mid \putl(u,x) \mid \overline{b}) \textsf{ in } r 
    \end{align*}
  \end{itemize}
  The second group of critical pairs is:
  \begin{itemize}
  \item[] (3-5): \vspace{-2em}
    \begin{align*}
     & \putl(v,\textsf{let } x_1 \seq x_n \asn (t \mid \overline{a} \mid \putr(u,s)) \textsf{ in } r)
      \\& \overset{\text{\bf[R3]}}{\from} \textsf{let } x_1 \seq x_n \asn (\putl(v,t) \mid \overline{a} \mid \putr(u,s)) \textsf{ in } r 
     \\& \overset{\text{\bf[R5]}}{\to} \putr(u,\textsf{let } x_1 \seq x_n \asn (\putl(v,t) \mid \overline{a} \mid s) \textsf{ in } r) 
    \end{align*}
  \item[] (3-6): \vspace{-2em}
    \begin{align*}
      & \putl(v,\textsf{let } x_1 \seq  x_n \asn (t \mid \overline{a} \mid \getr(x.s)) \textsf{ in } r) 
      \\& \overset{\text{\bf[R3]}}{\from} \textsf{let } x_1 \seq x_n \asn (\putl(v,t) \mid \overline{a} \mid \getr(x.s)) \textsf{ in } r 
      \\& \overset{\text{\bf[R6]}}{\to} \getr(x.\textsf{let } x_1 \seq x_n \asn (\putl(v,t) \mid \overline{a} \mid s) \textsf{ in } r) 
    \end{align*}
  \item[] (4-5): \vspace{-2em}
    \begin{align*}
      & \getl(y.\textsf{let } x_1 \seq x_n \asn (t \mid \overline{a} \mid \putr(v,s)) \textsf{ in } r) 
      \\& \overset{\text{\bf[R4]}}{\from}\textsf{let } x_1 \seq x_n \asn (\getl(y.t) \mid \overline{a} \mid \putr(v,s)) \textsf{ in } r
      \\& \overset{\text{\bf[R5]}}{\to} \putr(v,\textsf{let } x_1 \seq x_n \asn (\getl(y.t) \mid \overline{a} \mid s) \textsf{ in } r) 
    \end{align*}
  \item[] (4-6): \vspace{-2em}
    \begin{align*}
      & \getl(y.\textsf{let } x_1 \seq x_n \asn (t \mid \overline{a} \mid \getr(x.s)) \textsf{ in } r)
      \\& \overset{\text{\bf[R4]}}{\from} \textsf{let } x_1 \seq x_n \asn (\getl(y.t) \mid \overline{a} \mid \getr(x.s)) \textsf{ in } r
      \\& \overset{\text{\bf[R6]}}{\to} \getr(x.\textsf{let } x_1 \seq x_n \asn (\getl(y.t) \mid \overline{a} \mid s) \textsf{ in } r) 
    \end{align*}
  \end{itemize}
  The third and final group of critical pairs is:
  \begin{itemize}
  \item[] (5-7): \vspace{-2em}
    \begin{align*}
      & \putr(v,\textsf{let } x_1 \seq x_n \asn (\overline{a} \mid \putl(u,t)) \textsf{ in } r) 
      \\& \overset{\text{\bf[R5]}}{\from} \textsf{let } x_1 \seq x_n \asn (\overline{a} \mid \putr(v,\putl(u,t))) \textsf{ in } r
      \\& \overset{\text{\bf[R7]}}{\to} \textsf{let } x_1 \seq x_n \asn (\overline{a} \mid \putl(u,\putr(v,t))) \textsf{ in } r 
    \end{align*}
  \item[] (5-8): \vspace{-2em}
    \begin{align*}
      & \putr(v,\textsf{let } x_1 \seq x_n \asn (\overline{a} \mid \getl(x.t)) \textsf{ in } r) 
      \\& \overset{\text{\bf[R5]}}{\from} \textsf{let } x_1 \seq x_n \asn (\overline{a} \mid \putr(v,\getl(x.t))) \textsf{ in } r 
      \\& \overset{\text{\bf[R8]}}{\to}\textsf{let } x_1 \seq x_n \asn (\overline{a} \mid \getl(x.\putr(v,t))) \textsf{ in } r 
    \end{align*}
  \item[] (6-9): \vspace{-2em}
    \begin{align*}
      & \getr(x.\textsf{let } x_1 \seq x_n \asn (\overline{a} \mid \putl(v,t)) \textsf{ in } r) 
      \\& \overset{\text{\bf[R6]}}{\from} \textsf{let } x_1 \seq x_n \asn (\overline{a} \mid \getr(x.\putl(v,t))) \textsf{ in } r 
      \\& \overset{\text{\bf[R9]}}{\to} \textsf{let } x_1 \seq x_n \asn (\overline{a} \mid \putl(v,\getr(x.t))) \textsf{ in } r 
    \end{align*}
  \item[] (6-10): \vspace{-2em}
    \begin{align*}
      & \getr(x.\textsf{let } x_1 \seq x_n \asn (\overline{a} \mid \getl(y.t)) \textsf{ in } r) 
      \\& \overset{\text{\bf[R6]}}{\from} \textsf{let } x_1 \seq x_n \asn (\overline{a} \mid \getr(x.\getl(y.t))) \textsf{ in } r 
      \\& \overset{\text{\bf[R10]}}{\to} \textsf{let } x_1 \seq x_n \asn (\overline{a} \mid \getl(y. \getr(x.t))) \textsf{ in } r 
    \end{align*}
  \end{itemize}
  We will show how to join one class of critical pair from each of the three groups. The others in the same class are joinable similarly. We begin with the first group, from which we consider critical pairs of the form (1-7). There are two cases: either $\overline{a}$ is empty, or it is nonempty. If $\overline{a}$ is empty then the divergence is:
  \begin{align*}
    & \textsf{let } x_1 \seq x_n \asn (\putl(u,t) \mid s[v/x] \mid \overline{b}) \textsf{ in } r 
    \overset{\text{\bf[R1]}}{\from} \textsf{let } x_1 \seq x_n \asn (\putr(v,\putl(u,t)) \mid \getl(x.s) \mid \overline{b}) \textsf{ in } r 
    \\& \overset{\text{\bf[R7]}}{\to} \textsf{let } x_1 \seq x_n \asn (\putl(u,\putr(v,t)) \mid \getl(x.s) \mid \overline{b}) \textsf{ in } r 
  \end{align*}
  which is joinable as in:
  \begin{align*}
    & \textsf{let } x_1 \seq x_n \asn (\putl(u,t) \mid s[v/x] \mid \overline{b}) \textsf{ in } r
    \overset{\text{\bf[R3]}}{\to} \putl(u,\textsf{let } x_1 \seq x_n \asn (t \mid s[v/x] \mid \overline{b}) \textsf{ in } r)
    \\& \overset{\text{\bf[R1]}}{\from} \putl(u,\textsf{let } x_1 \seq x_n \asn (\putr(v,t) \mid \getl(x.s) \mid \overline{b}) \textsf{ in } r)
    \\& \overset{\text{\bf[R3]}}{\from} \textsf{let } x_1 \seq x_n \asn (\putl(u,\putr(v,t)) \mid \getl(x.s) \mid \overline{b}) \textsf{ in } r 
  \end{align*}
  If $\overline{a}$ is nonempty, then $\mathcal{L} = \textsf{let } x_1,\ldots,x_n \asn (\overline{a} \mid \blacksquare) \textsf{ in } r : \wiggle{U}{\blacksquare}{\Delta}{\blacksquare} \rightsquigarrow \wiggle{V}{\Gamma,\blacksquare}{B}{\blacksquare W}$ is an effective term context, and Lemma~\ref{lem:effectful-active} gives that there is an active term context $\mathcal{L}'$ (of the same type) such that $\mathcal{L} \stackrel{*}{\to} \mathcal{L}'$ and moreover $\mathcal{L}'[t_1 \seq t_n] = \mathcal{L}''[\getr(y.h),t_1,\ldots,t_n]$ for some $\mathcal{L}''$. Then the divergence in question is:
  \begin{align*}
    & \mathcal{L}[\putl(u,t),s[v/x],\overline{b}]
    \overset{\text{\bf[R1]}}{\from} \mathcal{L}[\putr(v,\putl(u,t)),\getl(x.s), \overline{b}]
    \overset{\text{\bf[R7]}}{\to} \mathcal{L}[\putl(u,\putr(v,t)),\getl(x.s),\overline{b}]
  \end{align*}
which is joinable as in:
\begin{align*}
  & \mathcal{L}[\putl(u,t) , s[v/x] , \overline{b}]
    \stackrel{*}{\to} \mathcal{L}'[\putl(u,t) , s[v/x] , \overline{b}]
    = \mathcal{L}''[\getr(y.h) , \putl(u,t) , s[v/x] , \overline{b}]
  \\& \overset{\text{\bf[R2]}}{\to} \mathcal{L}''[h[t/y] , t , s[v/x] , \overline{b}]
  \overset{\text{\bf[R1]}}{\from} \mathcal{L}''[h[u/y] ,\putr(v,t) , \getl(x.s) , \overline{b}]
  \\& \overset{\text{\bf[R2]}}{\from} \mathcal{L}''[\getr(y.h) , \putl(u,\putr(v,t)) , \getl(x.s) , \overline{b}]
  = \mathcal{L}'[\putl(u,\putr(v,t)) , \getl(x.s) , \overline{b}]
  \\& \stackrel{*}{\from} \mathcal{L}[\putl(u,\putr(v,t)) , \getl(x.s) , \overline{b}]
\end{align*}
Thus, all critical pairs of the form (1-7) are joinable. The rest of the critical pairs from the first group are joinable in a similar fashion. 

From the second group, we show how to join critical pairs of the form (3-5). The divergence is:
\begin{align*}
  & \putl(v,\textsf{let } x_1 \seq x_n \asn (t \mid \overline{a} \mid \putr(u,s)) \textsf{ in } r)
  \overset{\text{\bf[R3]}}{\from} \textsf{let } x_1 \seq x_n \asn (\putl(v,t) \mid \overline{a} \mid \putr(u,s)) \textsf{ in } r
  \\& \overset{\text{\bf[R5]}}{\to} \putr(u,\textsf{let } x_1 \seq x_n \asn (\putl(v,t) \mid \overline{a} \mid s) \textsf{ in } r)
\end{align*}
which is joinable as in:
\begin{align*}
  & \putl(v,\textsf{let } x_1 \seq x_n \asn (t \mid \overline{a} \mid \putr(u,s)) \textsf{ in } r)
    \overset{\text{\bf[R5]}}{\to} \putl(v,\putr(u,\textsf{let } x_1 \seq x_n \asn (t \mid \overline{a} \mid s) \textsf{ in } r))
  \\& \overset{\text{\bf[R7]}}{\from} \putr(u,\putl(v,\textsf{let } x_1 \seq x_n \asn (t \mid \overline{a} \mid s) \textsf{ in } r))
  \overset{\text{\bf[R3]}}{\from} \putr(u,\textsf{let } x_1 \seq x_n \asn (\putl(v,t) \mid \overline{a} \mid s) \textsf{ in } r)
\end{align*}
The rest of the critical pairs from the second group are joinable in a similar fashion. 

From the third group, we show how to join critical pairs of the form (5-7). If $\overline{a}$ is the empty sequence then the divergence in question is:
\begin{align*}
  & \putr(v,\textsf{let } x \asn (\putl(u,t)) \textsf{ in } r)
    \overset{\text{\bf[R5]}}{\from} \textsf{let } x \asn (\putr(v,\putl(u,t))) \textsf{ in } r
    \overset{\text{\bf[R7]}}{\to} \textsf{let } x \asn (\putl(u,\putr(v,t))) \textsf{ in } r
\end{align*}
which is joinable as in:
\begin{align*}
  & \putr(v,\textsf{let } x \asn (\putl(u,t)) \textsf{ in } r)
    \overset{\text{\bf[R3]}}{\to} \putr(v,\putl(u,\textsf{let } x \asn (t) \textsf{ in } r))
    \overset{\text{\bf[R7]}}{\to} \putl(u,\putr(v,\textsf{let } x \asn (t) \textsf{ in } r))
  \\& \overset{\text{\bf[R5]}}{\from} \putl(u,\textsf{let } x \asn (\putr(v,t)) \textsf{ in } r)
  \overset{\text{\bf[R3]}}{\from} \textsf{let } x \asn (\putl(u,\putr(v,t))) \textsf{ in } r
\end{align*}
If $\overline{a}$ is nonempty, then $\mathcal{L} = \textsf{let } x_1,\ldots,x_n \asn (\overline{a} \mid \blacksquare) \textsf{ in } r : \wiggle{U}{\blacksquare}{\Delta}{\blacksquare} \rightsquigarrow \wiggle{V}{\Gamma,\blacksquare}{B}{\blacksquare W}$ is an effective term context. Lemma~\ref{lem:effectful-active} gives that $\mathcal{L} \stackrel{*}{\to} \mathcal{L}'$ for some active $\mathcal{L}'$, and moreover that $\mathcal{L}'[t_1 \seq t_n] = \mathcal{L}''[\getr(y.s) , t_1 \seq t_n]$ for some $\mathcal{L}''$. Then the divergence is question is:
\begin{align*}
  & \putr(v,\mathcal{L}[\putl(u,t)])
    \overset{\text{\bf[R5]}}{\from} \mathcal{L}[\putr(v,\putl(u,t))]
    \overset{\text{\bf[R7]}}{\to} \mathcal{L}[\putl(u,\putr(v,t))]
\end{align*}
Now Lemma~\ref{lem:boundary-multipop} gives:
\[
\putr(v,\mathcal{L}''[s[u/y], t])
  \stackrel{*}{\to} h
  \stackrel{*}{\from} \mathcal{L}''[s[u/y] , \putr(v,t)]
\]
and so the divergence in question is joinable as in:
\begin{align*}
  & \putr(v,\mathcal{L}[\putl(u,t)])
    \stackrel{*}{\to} \putr(v,\mathcal{L}'[\putl(u,t)])
    = \putr(v,\mathcal{L}''[\getr(y.s), \putl(u,t)])
  \\& \overset{\text{\bf[R2]}}{\to} \putr(v,\mathcal{L}''[s[u/y], t])
  \stackrel{*}{\to} h
  \stackrel{*}{\from} \mathcal{L}''[s[u/y] , \putr(v,t)]
  \overset{\text{\bf[R2]}}{\from} \mathcal{L}''[\getr(y.s) , \putl(u,\putr(v,t))]
  \\& = \mathcal{L}'[\putl(u,\putr(v,t))]
  \stackrel{*}{\from} \mathcal{L}[\putl(u,\putr(v,t))]
\end{align*}
Thus, all critical pairs of the form (5-7) are joinable. The rest of the critical pairs in the third group are joinable similarly. Since all of the critical pairs are joinable, $\to$ is locally confluent.
\end{proof}

Now, from Lemma~\ref{lem:termination} and Lemma~\ref{lem:local-confluence} we have:
\begin{theorem}\label{thm:confluence}
  $\to$ is confluent.
\end{theorem}

The reader may have noticed that our rewrite relation $\to$ and term contexts $\mathcal{L}$ are ``left-biased''. This is in service of confluence: we must pick which of the left-facing and right-facing puts gets and puts are to be resolved first when both are available. Of course, we could have made the opposite choice and chosen to work with a ``right-biased'' rewrite relation, in which case the appropriate analogue of term contexts is similarly dual. This also yields a confluent and terminating rewrite relation. While it does not really matter which system we choose, we must choose one, and have here chosen the left-biased version.

\section{A Virtual Double Category of Terms}\label{sec:vdc}

In this section we show that the terms of our calculus form a virtual double category when considered modulo the convertibility relation $\stackrel{*}{\leftrightarrow}$ induced by our term rewrite relation $\to$. Having done this, we then show that our interpretation $\llbracket - \rrbracket$ of terms as cells of the free cornering defines a functor of virtual double categories. Our fixed multicategory $\mathcal{M}$ from Section~\ref{sec:term-calculus} remains fixed in this section.

The data of our virtual double category of terms is as follows:
\begin{definition}
  Define $[\mathcal{M}]$ to be the virtual double category that has one object, $*$, has horizontal morphisms $A : * \to *$ given by objects $A$ of $\mathcal{M}$, and has the category of vertical morphisms given by the monoid $\ex{\mathcal{M}}_0$, viewed as a category with one object. Cells $a : [\mathcal{M}]\cells{U}{A_1 ,\ldots,A_n}{B}{W}$ are $\stackrel{*}{\leftrightarrow}$-equivalence classes of derivations of the form $x_1 {:} A_1 \seq x_n {:} A_n \Vdash a : (U,B,W)$ (i.e., terms of $T(\mathcal{M})$). Identity cells $1_A : [\mathcal{M}]\cells{\lambda}{A}{A}{\lambda}$ are given by $x : A \Vdash [x] : (\lambda,A,\lambda)$. Composition is defined by:
  \[
    f \circ (g_1,\ldots,g_n) =
    \begin{cases}
      id_U(f) &\text{if }  (g_1,\ldots,g_n) = (\,)_U\\
      \textsf{let } x_1,\ldots,x_n \asn (g_1 \mid \cdots \mid g_n) \textsf{ in } f &\text{otherwise}
    \end{cases}
  \]
  where for $\Vdash f : (V,B,W)$ we define $\Vdash id_U(f) : (UV,B,UW)$ by induction on $U$ as in:
  \[
    id_U(f) =
    \begin{cases}
      f &\text{if } U = \lambda \\
      \getl(x.\putr(x.id_{U'}(f))) &\text{if } U = A^\circ U'\\
      \getr(x.\putl(x.id_{U'}(f))) &\text{if } U = A^\bullet U'
    \end{cases}
  \]
\end{definition}

We must show that the data of $[\mathcal{M}]$ satisfies the axioms of a virtual double category. To do this, we require a few technical lemmas concerning the behaviour of cells of the form $id_U(t)$. These show that $id_U(t)$ behaves like a ``horizontal identity'' when considered modulo our rewrite relation.
\begin{lemma}\label{lem:id-pop}
  For all $U \in \ex{\mathcal{M}}_0$, we have:
  \[
    \textsf{let } x_1 \seq x_n \asn (id_U(t_1) \mid \cdots \mid id_U(t_n)) \textsf{ in } t
    \,\stackrel{*}{\leftrightarrow}\, 
    id_U(\textsf{let } x_1 \seq x_n \asn (t_1 \mid \cdots \mid t_n) \textsf{ in } t)
  \]
  for all terms $t,t_1,\ldots,t_n$ for which this makes sense.
\end{lemma}
\begin{proof}
  We proceed by induction on the structure of $U$. The cases are as follows:
  \begin{itemize}
  \item If $U = \lambda$, then we have:
    \begin{align*}
      & \textsf{let } x_1 \seq x_n \asn (id_\lambda(t_1) \mid \cdots id_\lambda(t_n)) \textsf{ in } t
        = \textsf{let } x_1 \seq x_n \asn  (t_1 \mid \cdots \mid t_n) \textsf{ in } t
        \\&= id_\lambda(\textsf{let } x_1 \seq x_n \asn (t_1 \mid \cdots \mid t_n) \textsf{ in } t)
    \end{align*}
  \item If $U = A^\circ U'$, then we have:
    \begin{align*}
      & \textsf{let } x_1 \seq x_n \asn (id_{A^\circ U'}(t_1) \mid \cdots \mid id_{A^\circ U'}(t_n)) \textsf{ in } t
      \\& = \textsf{let } x_1 \seq x_n \asn (\getl(y_1. \putr(y_1, id_{U'}(t_1))) \mid \cdots \mid \getl(y_n.\putr(y_n,id_{U'}(t_n)))) \textsf{ in } t
      \\& \to \getl(y_1. \textsf{let } x_1 \seq x_n \asn ( \putr(y_1, id_{U'}(t_1)) \mid \cdots \mid \getl(y_n.\putr(y_n,id_{U'}(t_n)))) \textsf{ in } t)
      \\& \stackrel{*}{\to} \getl(y_1. \textsf{let } x_1 \seq x_n \asn (id_{U'}(t_1) \mid \cdots \mid \putr(y_n,id_{U'}(t_n))[y_1/y_n]) \textsf{ in } t)
      \\&= \getl(y_1. \textsf{let } x_1 \seq x_n \asn (id_{U'}(t_1) \mid \cdots \mid \putr(y_1,id_{U'}(t_n))) \textsf{ in } t)
      \\& \to \getl(y_1. \putr(y_1, \textsf{let } x_1 \seq x_n \asn (id_{U'}(t_1) \mid \cdots \mid id_{U'}(t_n)) \textsf{ in } t))
      \\& \stackrel{*}{\to} \getl(y_1. \putr(y_1, id_{U'}(\textsf{let } x_1 \seq x_n \asn (t_1 \mid \cdots \mid t_n) \textsf{ in } t)))
      \\& = id_{A^\circ U'}(\textsf{let } x_1 \seq x_n \asn (t_1 \mid \cdots \mid t_n) \textsf{ in } t)
    \end{align*}
  \item The case where $U = A^\bullet U'$ is similar.
  \end{itemize}
\end{proof}

\begin{lemma}\label{lem:interact-across-id}
  We have:
  \begin{enumerate} 
  \item 
    $f \circ (\dots\! , \textsf{getR}(x.t) , id_{A^\bullet U}(h_1) \seq id_{A^\bullet U}(h_k) , \textsf{putL}(v,s) ,\! \dots)
    \stackrel{*}{\to} 
    f \circ (\dots\! , t[v/x] , id_{U}(h_1) \seq  id_{U}(h_k) , s , \! \dots)$
  \item 
    $f \circ (\dots \! , \textsf{putR}(v,t) , id_{A^\circ U}(h_1) \seq id_{A^\circ U}(h_k) , \textsf{getL}(x.s) ,\! \dots)
    \stackrel{*}{\to} 
    f \circ (\dots \! , t , id_{U}(h_1) \seq id_{U}(h_k) , s[v/x] ,\! \dots)$
  \end{enumerate}
  whenever these expressions make sense.
\end{lemma}
\begin{proof}
  \begin{enumerate}
  \item By induction on $k$. If $k = 0$, then \textbf{[R2]} gives:
    \[
      f \circ (\dots \! , \textsf{getR}(x.t) , \textsf{putL}(v,s) , \! \dots)
      \to f \circ (\dots \! , t[v/x] , s ,\! \dots)
    \]
    For the inductive case, if the claim holds for $k$ then we have:
    \begin{align*}
      & f \circ (\dots \!, \textsf{getR}(x.t) , id_{A^\bullet U}(h_1) \seq  id_{A^\bullet U}(h_k) , id_{A^\bullet U}(h) , \textsf{putL}(v,s) , \! \dots)
      \\& = f \circ (\dots \! , \textsf{getR}(x.t) , id_{A^\bullet U}(h_1) \seq  id_{A^\bullet U}(h_k) , \textsf{getR}(y.\textsf{putL}(y,id_{U}(h))) , \textsf{putL}(v,s) , \! \dots)
      \\& \to f \circ (\dots \! , \textsf{getR}(x.t) , id_{A^\bullet U}(h_1) \seq  id_{A^\bullet U}(h_k) , \textsf{putL}(v,id_{U}(h)) , s , \! \dots)
      \\& \stackrel{*}{\to} f \circ (\dots \! , t[v/x] , id_{U}(h_1) \seq  id_{U}(h_k), id_U(h) , s , \! \dots)
    \end{align*}
    and the top-level claim holds by induction.
  \item Similar to (i).
  \end{enumerate}
\end{proof}

\begin{lemma}\label{lem:pop-across-id}
  We have:
  \begin{enumerate}
  \item 
    $f \circ (id_{A^\bullet U}(h_1) \seq id_{A^\bullet U}(h_k) , \putl(v,t) , g_1 \seq g_n)
    \stackrel{*}{\to} 
    \putl(v,f \circ (id_{U}(h_1)  \seq id_{U}(h_k) , t , g_1 \seq g_n))$
  \item 
    $f \circ (id_{A^\circ U}(h_1) \seq id_{A^\circ U}(h_k) , \getl(x.t) , g_1 \seq g_n)
    \stackrel{*}{\to} 
    \getl(x.f \circ (id_{U}(h_1)  \seq  id_{U}(h_k) , t , g_1 \seq g_n))$
  \item 
    $f \circ (g_1 \seq g_n , \putr(v,t) , id_{A^\circ U}(h_{1}) \seq id_{A^\circ U}(h_k))
    \stackrel{*}{\to} 
    \putr(v,f \circ (g_1 \seq g_n , t , id_{U}(h_1) \seq id_{U}(h_k)))$
  \item 
    $f \circ (g_1 \seq g_n , \getr(x.t) , id_{A^\bullet U}(h_{1}) \seq id_{A^\bullet U}(h_k))
    \stackrel{*}{\to} 
    \getr(x.f \circ (g_1 \seq g_n , t , id_{U}(h_1)  \seq  id_{U}(h_k)))$
  \end{enumerate}
  whenever these expressions make sense.
\end{lemma}

\begin{proof}
  \begin{enumerate}
  \item By induction on $k$. If $k = 0$ then \textbf{[R3]} gives:
    \begin{align*}
      & f \circ (\putl(v,t),g_1 \seq g_n)
        \to \putl(v, f \circ (t,g_1 \seq g_n))
    \end{align*}
    For the inductive case, if the claim holds for $k$ then we have:
    \begin{align*}
      & f \circ (id_{A^\bullet U}(h_1) \seq id_{A^\bullet U}(h_k) , id_{A^\bullet U}(h) , \putl(v,t) , g_1 \seq g_n)
      \\& = f \circ (id_{A^\bullet U}(h_1) \seq id_{A^\bullet U}(h_k) , \getr(x.\putl(x, id_{U}(h))) , \putl(v,t) , g_1 \seq g_n)
      \\& \to f \circ (id_{A^\bullet U}(h_1) \seq id_{A^\bullet U}(h_k) , \putl(v, id_{U}(h)) , t , g_1 \seq g_n)
      \\& \stackrel{*}{\to} \putl(v, f \circ (id_{U}(h_1) \seq id_{U}(h_k) , id_{U}(h) , t , g_1 \seq g_n))
    \end{align*}
    and the top-level claim holds by induction.
  \item By induction on $k$. If $k = 0$ then \textbf{[R4]} gives:
    \begin{align*}
      & f \circ (\getl(x.t),g_1 \seq g_n) \to \getl(x.f \circ(t,g_1 \seq g_n))
    \end{align*}
    For the inductive case, if the claim holds for $k$ then we have:
    \begin{align*}
      & f \circ (id_{A^\circ U}(h_1) \seq id_{A^\circ U}(h_k) , id_{A^\circ U}(h) , \getl(x.t) , g_1 \seq g_n)
      \\& = f \circ (id_{A^\circ U}(h_1) \seq id_{A^\circ U}(h_k) , \getl(y.\putr(y, id_{U}(h))) , \getl(x.t) , g_1\seq g_n)
      \\& \stackrel{*}{\to} \getl(y.f \circ (id_{U}(h_1) \seq id_{U}(h_k) , \putr(y,id_{U}(h)) , \getl(x.t) , g_1 \seq g_n))
      \\& \to \getl(y. f \circ (id_{U}(h_1) \seq id_{U}(h_k) , id_{U}(h) , t[y/x] , g_1 \seq g_n))
      \\& = \getl(x. f \circ (id_{U}(h_1) \seq id_{U}(h_l) , id_{U}(h_k) , t , g_1 \seq g_n)
    \end{align*}
    and the top-level claim holds by induction.
  \item Similar to (i).
  \item Similar to (ii).
  \end{enumerate}
\end{proof}

\begin{lemma}\label{lem:id-property-a}
	Let $U,W \in \ex{M}_0$, and let $t$ be a term. Then we have $id_U(id_W(t)) = id_{UW}(t)$.
\end{lemma}
\begin{proof}
	We proceed by structural induction on the form of $U$. The cases are as follows:
		\begin{itemize}
			\item If $U = \lambda$, then we have $id_{UW}(t) = id_{W}(t) = id_{\lambda}(id_{W}(t)) = id_U(id_W(t))$.
			\item If $U = A^\circ U'$, then we have:
			\begin{align*}
				& id_{UW}(t)
				= id_{A^\circ U' W}(t)
				= \getl(x.\putr(x.id_{U'W}(t)))
				= \getl(x.\putr(x.id_{U'}(id_{W}(t))))
				\\&= id_{A^\circ U'}(id_{W}(t))
				= id_{UW}(t)
			\end{align*}
			\item The case where $U = A^\bullet U'$ is similar.
		\end{itemize}
\end{proof}

Technical lemmas in hand, we proceed to show that composition in $[\mathcal{M}]$ is associative:
\begin{lemma}\label{lem:associative}
  We have:
  \[
    f \circ (g_1 \circ (h_1^1 \seq h_1^{m_1}) \seq g_n \circ (h_n^1 \seq h_n^{m_n}))
    \stackrel{*}{\leftrightarrow}
    (f \circ  (g_1 \seq g_n)) \circ (h_1^1\seq h_1^{m_1} \seq h_n^1 \seq h_n^{m_n})
  \]
  whenever these expressions make sense.
\end{lemma}
\begin{proof}
  If $n = 0$ then there are no $g_i$ or $h_i^j$ and all lists are empty. There are some $U,W \in \ex{\mathcal{M}}_0$ denoting the vertical edges of the empty lists, and Lemma~\ref{lem:id-property-a} gives:
  \begin{align*}
    & f \circ (\,)_{UW} = id_{UW}(f) = id_U(id_W(f)) = id_U(f \circ (\,)_W) = (f \circ (\,)_W) \circ (\,)_U
  \end{align*}
  as required. 
  
  If $n > 0$ and each $m_i = 0$ then Lemma~\ref{lem:id-pop} gives:
  \begin{align*}
    & f \circ (g_1 \circ (\,)_U \seq g_n \circ (\,)_U)
      = f \circ (id_U(g_1) \seq id_U(g_n))
      = \textsf{let } y_1\seq y_n \asn (id_U(g_1) \mid \cdots \mid id_U(g_n)) \textsf{ in } f
    \\& \stackrel{*}{\to} id_U(\textsf{let } y_1 \seq y_n \asn (g_1 \mid \cdots \mid g_n) \textsf{ in } f)
    = (f \circ (g_1 \seq g_n)) \circ (\,)_U
  \end{align*}
  as required. 
  
  If $n > 0$ and $m_i > 0$ for at least one $1 \leq i \leq n$, then we proceed by induction on $\sum_{i=1}^n(\sum_{j=1}^{m_n} \#(h_i^j))$. The base case is when this is zero, which is clearly impossible under our assumption that $m_i > 0$ for at least one $i$. For the inductive case, we consider a number of subcases:
  \begin{itemize}
  \item If some $h_i^j$ is a let-binding then it reduces by Lemma~\ref{lem:let-reduce} and the claim follows by induction hypothesis.
  \item If $h_i^j = [v_i^j]$ for all $i,j$ then for all $1 \leq i \leq n$ we have
    \[
      g_i \circ (h_i^1 \seq h_i^{m_i}) \stackrel{*}{\to} g_i[v_i^1\seq v_i^{m_i}/x_i^1 \seq x_i^{m_i}]
    \]
  In particular, if $m_i > 0$ this holds via \textbf{[R1]}, as in:
  \begin{align*}
    & g_i \circ (h_i^1 \seq h_i^{m_i})
      = g_i \circ ([v_i^1] \seq [v_i^{m_i}])
      = \textsf{let } x_i^1 \seq x_i^{m_i} \asn ([v_i^1] \mid \cdots \mid [v_i^{m_i}]) \textsf{ in } g_i
      \to g_i[v_i^1\seq v_i^{m_i}/x_i^1 \seq x_i^{m_i}]
  \end{align*}
  and if $m_i = 0$ then typing constraints ensure that $g_i \circ (h_i^1 \seq h_i^{m_i}) = g_i \circ (\,)_\lambda$ and we have:
  \begin{align*}
    & g_i \circ (h_i^1 \seq h_i^{m_i}) = g_i \circ (\,)_\lambda = id_\lambda(g_i) = g_i
  \end{align*}
  where we understand $t[\overline{v}/\overline{x}]$ to be $t$ when $\overline{v}$ and $\overline{x}$ are empty. As required we have: 
  \begin{align*}
    & f \circ (g_1 \circ (h_1^1 \seq h_1^{m_1}) \seq g_n \circ (h_n^1 \seq h_n^{m_n}))
     \stackrel{*}{\to} f \circ (g_1[v_1^1 \seq v_1^{m_1} / x_1^1 \seq x_1^{m_1}] \seq g_n[v_n^1 \seq v_n^{m_n}/x_n^1 \seq x_n^{m_n}])
    \\& = \textsf{let } y_1 \seq y_n \asn (g_1[v_1^1 \seq v_1^{m_1} / x_1^1 \seq x_1^{m_1}] \mid \cdots \mid g_n[v_n^1 \seq v_n^{m_n}/x_n^1 \seq x_n^{m_n}]) \textsf{ in  } f
    \\&= (\textsf{let } y_1 \seq y_n \asn (g_1 \seq g_n) \textsf{ in } f)[v_1^1 \seq v_1^{m_1} \seq v_n^1 \seq v_n^{m_n}/x_1^1 \seq x_1^{m_1} \seq x_n^1 \seq x_n^{m_n}]
    \\& \stackrel{*}{\from} \textsf{let } x_1^1 \seq x_1^{m_1} \seq x_n^1 \seq x_n^{m_n} \asn ([v_1^1] \mid \cdots \mid [v_1^{m_1}] \mid \cdots \mid [v_n^1] \mid \cdots \mid [v_n^{m_n}]) \textsf{ in } (\textsf{let } y_1 \seq y_n \asn (g_1 \seq g_n) \textsf{ in } f)
    \\& = (f \circ (g_1 \seq g_n)) \circ ([v_1^1] \seq [v_1^{m_1}] \seq [v_n^1] \seq [v_n^{m_n}])
    = (f \circ (g_1 \seq g_n)) \circ (h_1^1 \seq h_1^{m_1} \seq h_n^1 \seq h_n^{m_n})
  \end{align*}
\item If no $h_i^j$ is a let-binding but $h_i^j$ for some $i$ and $j$ is left-facing, let $(i,j)$ be (lexicographically) the least pair of indices for which this is so. We consider a number of subcases: 

First, if $j = 1$ and $m_k = 0$ for all $1 \leq k < i$ then in the case where $h_i^j = \putl(v,t)$ Lemma~\ref{lem:pop-across-id} gives:
  \begin{align*}
    & f \circ (g_1 \circ (h_1^1 \seq h_1^{m_1}) \seq g_{i-1} \circ (h_{i-1}^1 \seq h_{i-1}^{m_{i-1}}),g_i \circ (h_i^1 \seq h_i^{m_i}) \seq g_n \circ (h_n^1 \seq h_n^{m_n}))
    \\&= f \circ (g_1 \circ (\,)_{A^\bullet U} \seq g_{i-1} \circ (\,)_{A^\bullet U},g_i \circ (\putl(v,t),\seq,h_i^{m_i}) \seq g_n\circ (h_n^1 \seq h_n^{m_n}))
    \\& \to f \circ (g_1 \circ (\,)_{A^\bullet U} \seq g_{i-1} \circ (\,)_{A^\bullet U}, \putl(v,g_i \circ (t,\seq,h_i^{m_i})) \seq g_n\circ (h_n^1 \seq h_n^{m_n}))
    \\& \stackrel{*}{\to} \putl(v,f \circ (g_1 \circ (\,)_U\seq g_{i-1} \circ (\,)_U , g_i \circ (t \seq h_i^{m_i})\seq g_n\circ (h_n^1 \seq h_n^{m_n})))
    \\& \stackrel{*}{\leftrightarrow} \putl(v, (f \circ (g_1 \seq g_{i-1},g_i \seq g_n)) \circ (t\seq h_i^{m_i}\seq h_n^1 \seq h_n^{m_n}))
    \\& \stackrel{*}{\from} (f \circ (g_1 \seq g_{i-1},g_i \seq g_n)) \circ (\putl(v,t)\seq h_i^{m_i} \seq h_n^1 \seq h_n^{m_n})
  \end{align*}
  as required. The case where $h_i^j = \getl(x.t)$ is similar. 
  
  Second, if we have $j=1$ and $m_k > 0$ for some $1 \leq k < i$, consider the largest such $k$. If $h_i^j = \putl(v,t)$ then typing constraints ensure $h_k^{m_k} = \getr(x.s)$ and Lemma~\ref{lem:interact-across-id} gives:
  \begin{align*}
    & f \circ (\ldots \!, g_k \circ (h_k^1 \seq h_k^{m_k}),g_{k+1} \circ (h_{k+1}^{1}\seq h_{k+1}^{m_{k+1}}) \seq g_{i-1} \circ (h_{i-1}^1\seq h_{i-1}^{m_{i-1}}),g_i \circ (h_i^1 \seq h_i^{m_i}) ,\! \ldots)
    \\&= f \circ (\ldots\!, g_k \circ (h_k^1 \seq \getr(x.s)),g_{k+1} \circ (\,)_{A^\bullet U} \seq g_{i-1} \circ (\,)_{A^\bullet U} ,g_i \circ (\putl(v,t) \seq h_i^{m_i}),\!\ldots)
    \\& \stackrel{*}{\to} f \circ (\ldots\!, \getr(x.g_k \circ (h_k^1 \seq s)),g_{k+1} \circ (\,)_{A^\bullet U} \seq g_{i-1} \circ (\,)_{A^\bullet U} ,\putl(v,g_i \circ (t \seq h_i^{m_i})),\!\ldots)
    \\& \stackrel{*}{\to} f \circ (\ldots\!, g_k \circ (h_k^1 \seq s[v/x]),g_{k+1} \circ (\,)_{U} \seq g_{i-1} \circ (\,)_{U} , g_i \circ (t \seq h_i^{m_i}),\!\ldots)
    \\& \stackrel{*}{\leftrightarrow} (f \circ (\ldots \!, g_k,g_{k+1}\seq g_{i-1},g_i ,\!\ldots)) \circ (\ldots\!, h_k^1 \seq s[v/x],t \seq h_i^{m_i} ,\!\ldots)
    \\& \from  (f \circ (\ldots \!, g_k,g_{k+1}\seq g_{i-1},g_i ,\!\ldots)) \circ (\ldots\!, h_k^1 \seq \getr(x.s), \putl(v,t) \seq h_i^{m_i} ,\!\ldots)
    \\&= (f \circ (\ldots \!, g_k,g_{k+1}\seq g_{i-1},g_i ,\!\ldots)) \circ (\ldots \!, h_k^1 \seq h_k^{m_k},h_{k+1}^1 \seq h_{k+1}^{m_{k+1}}\seq h_{i-1}^1 \seq h_{i-1}^{m_{i-1}},h_i^1 \seq h_i^{m_i} ,\!\ldots)
  \end{align*}
  as required. The case where $h_i^j = \getl(x.t)$ is similar. 
  
  Finally, suppose $j > 1$. If $h_i^j = \putl(v,t)$ then typing constraints, together with the assumption that $(i,j)$ is the smallest pair of indices for which $h_i^j$ is left-facing, ensure that $h_i^{j-1} = \getr(x.s)$ and we have:
  \begin{align*}
    & f \circ (g_1 \circ (h_1^1 \seq h_1^{m_1}) \seq g_i \circ (h_i^1 \seq h_i^{j-1},h_i^j \seq h_i^{m_i}) \seq g_n \circ (h_n^1 \seq h_n^{m_n}))
    \\&= f \circ (g_1 \circ (h_1^1 \seq h_1^{m_1}) \seq g_i \circ (h_i^1 \seq \getr(x.s),\putl(v,t) \seq h_i^{m_i}) \seq g_n \circ (h_n^1 \seq h_n^{m_n}))
    \\& \to f \circ (g_1 \circ (h_1^1 \seq h_1^{m_1}) \seq g_i \circ (h_i^1 \seq s[v/x],t \seq h_i^{m_i}) \seq g_n \circ (h_n^1 \seq h_n^{m_n}))
    \\& \stackrel{*}{\leftrightarrow} (f \circ (g_1 \seq g_i \seq g_n)) \circ (h_1^1 \seq h_1^{m_1} \seq h_i^1 \seq s[v/x],t \seq h_i^{m_i} \seq h_n^1 \seq h_n^{m_n})
    \\& \from (f \circ (g_1 \seq g_i \seq g_n)) \circ (h_1^1 \seq h_1^{m_1} \seq h_i^1 \seq \getr(x.s),\putl(v,t) \seq h_i^{m_i} \seq h_n^1 \seq h_n^{m_n})
    \\&= (f \circ (g_1 \seq g_i \seq g_n)) \circ (h_1^1 \seq h_1^{m_1} \seq h_i^1 \seq h_i^{j-1}, h_i^j \seq h_i^{m_i} \seq h_n^1 \seq h_n^{m_n})
  \end{align*}
  as required. The case where $h_i^j = \getl(x.t)$ is similar.
\item If no $h_i^j$ is left-facing or a let-binding and at least one $h_i^j$ is right-facing, let $(i,j)$ be (lexicographically) the greatest pair of indices for which this is so. Then $j = m_i$ since otherwise typing constraints force $h_i^{j+1}$ to be either left-facing or right-facing, both of which are contradictory. Similarly, for any $(i',j') > (i,j)$ the term $h_{i'}^{j'}$ cannot be left-facing or right-facing, but typing constrains ensure that it cannot be neutral, and it cannot be a let-binding by assumption. Thus, we know that $m_k = 0$ for all $i < k \leq n$. Then when $h_i^j = h_i^{m_i} = \putr(v,t)$ Lemma~\ref{lem:pop-across-id} gives:
  \begin{align*}
    & f \circ (g_1 \circ (h_1^1 \seq h_1^{m_1}) \seq g_i \circ (h_i^1 \seq h_i^{m_i}),g_{i+1} \circ (h_{i+1}^1 \seq h_{i+1}^{m_{i+1}}) \seq g_n \circ (h_n^1 \seq h_n^{m_n}))
    \\&= f \circ (g_1 \circ (h_1^1 \seq h_1^{m_1}) \seq g_i \circ (h_i^1 \seq \putr(v,t)),g_{i+1} \circ (\,)_{A^\circ U} \seq g_n \circ (\,)_{A^\circ U})
    \\& \stackrel{*}{\to} \putr(v, f \circ (g_1 \circ (h_1^1 \seq h_1^{m_1}) \seq g_i \circ (h_i^1 \seq t),g_{i+1} \circ (\,)_{U} \seq g_n \circ (\,)_{U}))
    \\& \stackrel{*}{\leftrightarrow} \putr(v, (f \circ (g_1 \seq g_i,g_{i+1} \seq g_n)) \circ (h_1^1\seq h_1^m \seq h_i^1 \seq t))
    \\& \from  (f \circ (g_1 \seq g_i,g_{i+1} \seq g_n)) \circ (h_1^1\seq h_1^m \seq h_i^1 \seq \putr(v,t))
    \\&= (f \circ (g_1 \seq g_i ,g_{i+1} \seq g_n)) \circ (h_1^1 \seq h_1^{m_1} \seq h_i^1 \seq h_i^{m_i} , h_{i+1}^1 \seq h_{i+1}^{m_{i+1}} \seq h_n^1 \seq h_n^{m_n})
  \end{align*}
  as required. The case where $h_i^j = h_i^{m_i} = \getr(x.t)$ is similar.
\end{itemize}
\end{proof}

As promised, we have:
\begin{theorem}\label{thm:vdc-structure}
  $[\mathcal{M}]$ is a virtual double category.
\end{theorem}
\begin{proof}
	Lemma \ref{lem:associative} shows that composition in $[\mathcal{M}]$ is associative. 
	It remains to show that the unit axioms hold. We begin by showing that $1_A \circ (f) \stackrel{*}{\leftrightarrow} f$ whenever this makes sense. We proceed by structural induction on $f$. The base case is when $f = [v]$, in which case:
	\[
	1_A \circ (f) 
	= 1_A \circ ([v])
	= \textsf{let } x \asn ([v]) \textsf{ in } [x]
	\to [x[v/x]]
	= [v]
	= f
	\]
	In case $f = \textsf{let } x_1 \seq x_n \asn (t_1 \mid \cdots \mid t_n) \textsf{ in } t$ then associativity of composition gives:
	\begin{align*}
		& 1_A \circ (f)
		= 1_A \circ (\textsf{let } x_1 \seq x_n \asn (t_1 \mid \cdots \mid t_n) \textsf{ in } t)
		= 1_A \circ (t \circ (t_1 \seq t_n))
		= (1_A \circ (t)) \circ (t_1 \seq t_n)
		\\&\stackrel{*}{\leftrightarrow} t \circ (t_1 \seq t_n)
		= \textsf{let } x_1\seq x_n \asn (t_1 \mid \cdots \mid t_n) \textsf{ in } t
		= f
	\end{align*}
	In case $f = \putr(v,t)$ we have:
	\begin{align*}
		& 1_A \circ (f)
		\asn (\textsf{let } x \asn (\putr(v,t)) \textsf{ in  } [x])
		= \putr(v,\textsf{let } x \asn (t) \textsf{ in } [x])
		= \putr(v,1_A \circ (t))
		\stackrel{*}{\leftrightarrow} \putr(v,t)
	\end{align*}
	and the cases where $f = \getr(y.t)$, $f = \putr(v,t)$, and $f = \getl(y.t)$ are similar. The left unit law follows by induction. Finally, we have $f \circ (1_{A_1} \seq 1_{A_n}) \stackrel{*}{\leftrightarrow} f$ whenever this makes sense as in:
	\[
	f \circ (1_{A_1} \seq 1_{A_n})
	= \textsf{let } x_1 \seq x_n \asn ([x_1] \mid \cdots \mid [x_n]) \textsf{ in } f
	\to f[x_1\seq x_n / x_1 \seq x_n]
	= f
      \]
      We conclude that $[\mathcal{M}]$ is a virtual double category.
\end{proof}

We proceed to show that our interpretation of terms of $T(\mathcal{M})$ as cells of $\corner{F(\mathcal{M})}$ (Definition~\ref{def:interpretation}) defines a functor of virtual double categories from $[\mathcal{M}]$ into the underlying virtual double category of $\corner{F(\mathcal{M})}$. If $\X$ is a single-object double category with horizontal edge monoid $(\X_H,\otimes,I)$ then the virtual double category $U(\X)$ has cells $a : \cells{U}{A_1,\ldots,A_n}{B}{W}$ given by cells $a : \cells{U}{A_1 \otimes \cdots \otimes A_n}{B}{W}$ of $\X$, with composition defined in terms of vertical and horizontal composition in $\X$ as in $f \circ (g_1 \seq g_n) = (g_1 \mid \cdots \mid g_n) \cdot f$, and with identities given by vertical identities in $\X$ (see e.g.,~\cite{Leinster2004}). 

To show we have a functor, we require one final technical lemma:
\begin{lemma}\label{lem:id-property}
  Let $U,W \in \ex{M}_0$, and let $t$ be a term. Then we have $\llbracket id_U(t) \rrbracket = id_U \cdot \llbracket t \rrbracket$.
\end{lemma}
\begin{proof}
  We proceed by induction on the form of $U$. The cases are as follows:
  \begin{itemize}
  \item If $U = \lambda$, then we have $\llbracket id_U(t) \rrbracket = \llbracket id_\lambda(t) \rrbracket = \llbracket t \rrbracket = id_\lambda \cdot \llbracket t \rrbracket = id_U \cdot \llbracket t \rrbracket$.
  \item If $U = A^\circ U'$, then using a yanking equation we have:
    \begin{align*}
      & \llbracket id_U(t) \rrbracket
        = \llbracket id_{A^\circ U'}(t) \rrbracket
        = \llbracket \getl(x.\putr(x.id_{U'}(t))) \rrbracket
        = {A\urcorner} \cdot \llbracket \putr(x.id_{U'}(t)) \rrbracket
      \\&= {A\urcorner} \cdot {A\llcorner} \cdot \llbracket id_{U'}(t) \rrbracket
      = id_{A^\circ} \cdot \llbracket id_{U'}(t) \rrbracket
      = id_{A^\circ} \cdot id_{U'} \cdot \llbracket t \rrbracket
      = id_{A^\circ U'} \cdot \llbracket t \rrbracket
      = id_U \cdot \llbracket t \rrbracket
    \end{align*}
  \item The case where $U = A^\bullet U'$ is similar.
	\end{itemize}
\end{proof}

Finally, we have:
\begin{theorem}\label{thm:functor}
  The interpretation $\llbracket - \rrbracket$ of terms defines a functor of virtual double categories:
  \[ \llbracket - \rrbracket : [\mathcal{M}] \to U\left(\corner{F(\mathcal{M})}\right) \]
\end{theorem}
\begin{proof}
  That $\llbracket - \rrbracket$ is well-defined on $\stackrel{*}{\leftrightarrow}$-equivalence classes of terms follows from Lemma~\ref{lem:interpretation-coherent}. We must show that $\llbracket - \rrbracket$ preserves composition and identities. For identities, we have that: \[ \llbracket 1_A \rrbracket = \llbracket x : A \Vdash [x] : (\lambda,A,\lambda) \rrbracket = \corner{(x : A \vdash x : A)} = \corner{(1_A)} = \corner{1_{(A)}} = 1_{(A)} \]
  For composition, there are two cases. First, Lemma~\ref{lem:id-property} gives:
  \[ \llbracket f \circ (\,)_U\rrbracket = \llbracket id_U(f)\rrbracket = id_U \cdot \llbracket f \rrbracket = \llbracket f \rrbracket \circ (\,)_U \]
  Second, when $n \geq 1$ we have:
  \begin{align*}
    & \llbracket f \circ (g_1 \seq g_n) \rrbracket
      = \llbracket \textsf{let } x_1 \seq x_n \asn (g_1 \mid \cdots \mid g_n) \textsf{ in } f \rrbracket
      = (\llbracket g_1 \rrbracket \mid \cdots \mid \llbracket g_n \rrbracket) \cdot \llbracket f \rrbracket
      = \llbracket f \rrbracket \circ (\llbracket g_1 \rrbracket \seq \llbracket g_n \rrbracket)
  \end{align*}
\end{proof}

\section{Concluding Remarks}\label{sec:conclusion}
We have introduced a calculus of interacting processes, consisting of a collection of terms together with a term rewrite relation. The calculus is parameterised by an arbitrary multicategory of non-interacting processes. We have shown that the calculus is confluent and terminating, and that terms of the calculus form a virtual double category when considered modulo the induced convertibility relation.

We imagine two primary directions for future work. First, we believe, but cannot yet prove, that the functor of virtual double categories given in Theorem~\ref{thm:functor} is faithful. If we think of the rewrite relation of our calculus in terms of operational semantics, then this functor gives a sound denotational semantics of our notion of process interaction. For the interpretation functor to be faithful is for our denotational semantics to be \emph{adequate} (see e.g.,~\cite{Cardone2021}). Adequacy is, essentially, the property that the denotational and operational semantics coincide in an appropriate sense. An adequate denotational semantics is much more useful than one that is merely sound, and as such we would like to be able to prove that the functor in question is faithful.

Second, process interaction in our calculus is governed by what are essentially session types. Viewed from this perspective, the calculus is missing certain features. For example, it is unable to express branching protocols. The free cornering construction has been extended with the ability to express branching protocols and protocol iteration~\cite{Nester2024}, and we imagine that the ideas developed in that setting could be applied to our calculus in order to increase its expressiveness. We hope that a more expressive version of the calculus presented here could be used to reason about e.g., cryptographic protocols.

\bibliographystyle{entics}
\bibliography{citations}

\end{document}